\newif\ifSubmission \Submissionfalse

\ifSubmission
  \documentclass[a4paper,american,reqno,12pt]{amsart}
  \usepackage[margin=1in]{geometry}
  \usepackage{setspace}
  
\else
  \documentclass[a4paper,american,reqno]{amsart}
\fi

\newif\ifPreprint \Preprinttrue

\usepackage{babel}
\usepackage[utf8]{inputenc}
\usepackage[T1]{fontenc}
\usepackage[binary-units=true]{siunitx}
\usepackage{booktabs}
\usepackage{enumitem}
\usepackage{siunitx}
\usepackage[ruled,vlined]{algorithm2e}
\usepackage{pgfplots}
\usepackage{todonotes}
\usepackage{graphicx}
\usepackage{float}
\usepackage[style = authoryear-comp,
            maxbibnames = 100,
            maxcitenames = 2,
            giveninits = true,
            uniquename = init,
            isbn = false,
            backend = bibtex]{biblatex}
\usepackage[colorlinks,
            citecolor=blue,
            urlcolor=blue,
            linkcolor=blue]{hyperref} 

\newcommand{\st}{\text{s.t.}}

\newcommand{\Defset}[3][\defsep]{\Set{#2#1#3}}
\newcommand{\set}[1]{\{#1\}}
\newcommand{\Set}[1]{\left\{#1\right\}}

\newcommand{\Abs}[1]{\left\lvert#1\right\rvert}
\newcommand{\define}{\mathrel{{\mathop:}{=}}}

\newtheorem{theorem}{Theorem}
\newtheorem{lemma}{Lemma}
\newtheorem{example}{Example}
\newtheorem{corollary}{Corollary}
\newtheorem{assumption}{Assumption}
\newtheorem{proposition}{Proposition}

\newtheorem{definition}{Definition}
\newtheorem{remark}{Remark}

\newcommand{\arcs}{A}
\newcommand{\edges}{E}
\newcommand{\inarcs}[1]{\delta^{\text{in}}(#1)}
\newcommand{\outarcs}[1]{\delta^{\text{out}}(#1)}
\newcommand{\nodes}{V}
\newcommand{\flow}{f}
\newcommand{\comflow}{x}
\newcommand{\origs}{S}
\newcommand{\dests}{T}
\newcommand{\commodities}{K}
\newcommand{\numcommodities}{|K|}
\newcommand{\latency}{L}

\newcommand{\uncertaintyset}{\ensuremath{\mathcal{U}}}
\newcommand{\bpr}{\text{BPR}}
\newcommand{\fix}{\text{fix}}
\newcommand{\capacity}{\ensuremath{u}}
\newcommand{\diff}{\xspace\,\mathrm{d}}

\definecolor{my-red}{HTML}{d7191c}
\definecolor{my-green}{HTML}{abdda4}
\def\cross{\textcolor{my-red}{\textsf{X}}}
\renewcommand{\check}{\textcolor{my-green!70!black}{\checkmark}}

\newcommand{\vv}[1]{\ifcat#1\relax\bm{#1}\else\mathbf{#1}\fi}

\newcommand{\field}{\mathbb}

\newcommand{\reals}{\field{R}}

\newcommand{\R}{\reals}

\bibliography{congestion-w-wardrop}

\begin{document}

\title[Modeling Network Congestion Using Wardrop Principles]{Modeling Network
  Congestion under Demand Uncertainty Using Wardrop Principles}

\author[Y. Beck]{Yasmine Beck}
\author[F. Giancola]{Francesca Giancola}
\author[I. Ljubi\'c]{Ivana Ljubi\'c}
\author[S. Mattia]{Sara Mattia}

\address[Y. Beck]{%
  Eindhoven University of Technology,
  Department of Industrial Engineering and Innovation Sciences,
  PO Box 513,
  5600 MB Eindhoven,
  the Netherlands}%
\email{y.beck@tue.nl}

\address[F. Giancola]{%
  (A) Istituto di Analisi dei Sistemi ed Informatica ``Antonio Ruberti'',
  Consiglio Nazionale delle Ricerche,
  Via dei Taurini, 19,
  00185 Rome,
  Italy;
  (B) Dipartimento di Ingegneria Informatica, Automatica e Gestionale,
  Sapienza Universit\`a di Roma,
  Via Ariosto, 25,
  00185 Rome,
  Italy}%
\email{francesca.giancola@iasi.cnr.it}

\address[I. Ljubi\'c]{%
  ESSEC Business School,
  Department of Information Systems, Data Analytics and Operations,
  3 Avenue Bernard Hirsch,
  95021 Cergy-Pontoise Cedex,
  France}%
\email{ljubic@essec.edu}

\address[S. Mattia]{%
  Istituto di Analisi dei Sistemi ed Informatica ``Antonio Ruberti'',
  Consiglio Nazionale delle Ricerche,
  Via dei Taurini, 19,
  00185 Rome,
  Italy}%
\email{sara.mattia@iasi.cnr.it}

\date{\today}


\begin{abstract}
  Motivated by the need for reliable traffic management under fluctuating travel
demand, we study the problem of determining the worst-case congestion in a
multi-commodity traffic network subject to demand uncertainty.
To this end, we stress-test a given network by identifying demand realizations
and corresponding travelers' route choices that maximize congestion.
The users of the traffic network are assumed to act according to one of the
two Wardrop principles---the user equilibrium or the system optimum---so that
the resulting congestion models can be seen as bilevel problems with a single
leader and multiple followers.
To address uncertain travel demand, we consider different models
such as ellipsoidal or budgeted uncertainty sets and the hose polyhedron.
We present single-level mixed-integer nonlinear reformulations of the
congestion models that exploit binary variables and big-$M$ constants, prove
the existence of optimal solutions, derive valid big-$M$s, and propose several
enhancement techniques to further strengthen the formulations. 
An extensive computational study on instances of the Sioux Falls network and
instances from the \textsf{SNDlib} demonstrates the computational
effectiveness of the proposed techniques and provides insight into the impact
of different congestion measures and uncertainty models on the resulting
worst-case congestion.


\end{abstract}

\keywords{Traffic networks, Wardrop principles,
  Bilevel optimization, Mixed-integer nonlinear optimization}
\subjclass[2020]{90B20,90C33,90C35,90C70}

\maketitle


\section{Introduction}
\label{sec:introduction}

Congestion is one of the central challenges in urban transportation,
leading to increased travel times, fuel consumption, and emissions.
To mitigate these effects, effective traffic management requires
anticipating how travelers respond to changing network conditions,
especially variations in travel demand arising from, e.g., daily
fluctuations, seasonal trends, or unforeseen events.
Because such variations are inevitable and often difficult to predict
accurately, understanding the potential worst-case congestion in a traffic
network is crucial.

In this paper, we study the problem of determining the worst-case congestion in
a multi-commodity traffic network subject to demand uncertainty. Effectively,
we stress-test a given network by identifying demand realizations and
corresponding travelers' route choices that maximize congestion.
To this end, we propose a novel bilevel formulation in which a traffic planner
acts as the leader and the users of the traffic network act as the followers.
In the leader's problem, we explicitly model the uncertain travel demand
against which the traffic planner wants to hedge. For this purpose, we
exploit ideas from robust optimization
\parencite{Soyster:1973,Ben-Tal_et_al:2009,Bertsimas_et_al:2011} so that demand
realizations are drawn from a predefined uncertainty set.
We use both ellipsoidal and polyhedral uncertainty sets,
including the hose polyhedron \parencite{hose,hose2} and the budgeted
uncertainty set \parencite{Bertsimas_Sim:2003,Bertsimas_Sim:2004,Sim:2004}.
To model travelers' route choices, we consider the two Wardrop
principles---the user equilibrium and the system optimum
\parencite{Wardrop_Whitehead:1952,Wardrop:1952}.
In a user equilibrium (UE), travelers select routes that minimize their
individual travel costs so that no traveler can improve their travel time by
unilaterally changing routes.
In contrast, under the system optimum (SO), a central planner coordinates or
assigns traffic such as to minimize the total travel time across the network.
Overall, we thus consider a single-leader multi-follower bilevel problem in
which Wardrop principles govern the followers' decisions.
An accessible overview of the two Wardrop principles can, e.g., be found in
\textcite{Correa_Stier-Moses:2011}. Moreover, we refer to
\textcite{Dempe:2002,Bard:1998,Dempe_et_al:2015},
and the surveys in
\textcite{Colson-et-al:2007,Colson_et_al:2005,Kleinert_et_al:2021}
for a general overview of bilevel optimization.
Let us further mention that bilevel problems are notoriously hard to
solve. \textcite{Hansen_et_al:1992} have shown that even linear bilevel
problems, i.e., those with continuous variables, linear objective functions,
and linear constraints, are strongly NP-hard in~general.

Building on the modeling choices described above and acknowledging the
computational challenges of bilevel optimization, the congestion models studied
in this paper are reformulated as mixed-integer nonlinear problems that
exploit binary variables and big-$M$ constants, which can be
tackled using state-of-the-art general-purpose solvers.
We prove the existence of optimal solutions to these problems, derive
valid big-$M$s, and propose several enhancement techniques that exploit the
network topology and structural properties of optimal flows to further
strengthen the formulations. 
To measure congestion, we consider three different functions: the maximum
utilization function, which focuses on the most congested arc; the total
utilization function, which accounts for congestion on all arcs of the network;
and the Bureau of Public Roads (BPR) function, which penalizes arcs with very
high congestion levels.
To compare the different variants of the congestion model and to assess the
effectiveness of the proposed enhancement techniques, we conduct an extensive
computational study on instances of the Sioux Falls network
\parencite{LeBlanc_et_al:1975} and instances from the \textsf{SNDlib}
\parencite{Orlowski_et_al:2010,SNDlib10}.
We analyze both user equilibrium and system optimum formulations in terms of
congestion levels and computational performance. Moreover, we compare
different problem variants arising from the choice of congestion measures and
uncertainty models, as well as from different travel cost functions, including
linear, quadratic, and BPR cost functions.
Our computational study reveals several key insights.
First, the proposed enhancement techniques significantly improve computational 
performance, enabling the solution of larger instances that would otherwise
remain unsolved.
Second, the choice of the travel cost function has a greater impact on the
computational tractability of the congestion models than the choice of the
congestion measure.
Third, the UE formulation can be solved considerably faster than 
the SO formulation, while congestion levels under the UE are typically
close to those under the SO.
Fourth and finally, we observe a trade-off between worst-case congestion and
computational cost across different uncertainty models, as congestion
estimates over larger uncertainty sets are typically obtained at increased
computational effort.

\subsection*{Related Literature}

Many applications of bilevel optimization arise in the context of
transportation; see, e.g., \textcite{Migdalas:1995} for a general overview.
In particular, bilevel models have been studied in
network design \parencite{LeBlanc_Boyce:1986,Marcotte:1986,%
  Ben-Ayed_et_al:1988,Fontaine_Minner:2014,Rey_Levin:2024},
network pricing \parencite{Labbe_et_al:1998,Brotcorne_et_al:2001,%
  Kalashnikov_et_al:2020,Dempe_Zemkoho:2012,Dewez_et_al:2008,Beck_et_al:2024},
or pricing with routing \parencite{Cerulli_et_al:2024}. 
Among these contributions, only a few works consider bilevel
formulations in which the lower level is modeled using Wardrop principles.
To the best of our knowledge,
\textcite{Beck_et_al:2024,Dempe_Zemkoho:2012,Rey_Levin:2024} are the closest
related studies in this context.
In \textcite{Rey_Levin:2024}, the authors study the discrete network design
problem in which travelers' act according to the Wardrop user equilibrium
principle.
To solve the resulting bilevel model, the authors present a
branch-and-price-and-cut approach.
In \textcite{Dempe_Zemkoho:2012}, the authors consider a bilevel toll-setting
problem in which travelers again follow the Wardrop user equilibrium
principle. In their work, the authors focus primarily on the theoretical
properties of the model, but a closely related toll-setting problem is studied
in \textcite{Beck_et_al:2024} from a computational point of view. In addition,
\textcite{Beck_et_al:2024} incorporate robust Wardrop equilibria to model
travelers' route choices subject to uncertain travel costs.
Overall, our work thus differs from the above contributions in two main
aspects.
First, while \textcite{Rey_Levin:2024} address network design and
\textcite{Dempe_Zemkoho:2012,Beck_et_al:2024} focus on network pricing, we study
the problem of determining the worst-case congestion that can arise in a given
traffic network.
Second, we explicitly account for uncertainty in the travel demand, whereas
\textcite{Beck_et_al:2024} consider uncertainty in the travel costs and
\textcite{Dempe_Zemkoho:2012,Rey_Levin:2024} study deterministic settings.
To the best of our knowledge, no existing work considers the problem of
stress-testing a traffic network under uncertain travel demand using a bilevel
formulation with Wardrop principles governing the travelers' behavior.
Although we focus on existing traffic networks, let us mention that the
insights provided by our models may also be useful in network design contexts,
for example to identify which parts of a network are most susceptible to
congestion under uncertain travel demand.

\subsection*{Outline}

The remainder of this paper is organized as follows.
In Section~\ref{sec:problem-statement}, we define the congestion models under
the user equilibrium and the system optimum.
In Section~\ref{sec:reformulations}, we present mixed-integer nonlinear
reformulations of these problems that exploit binary variables and big-$M$
constants and prove the existence of valid big-$M$s and optimal
solutions.
In Section~\ref{sec:tightening}, we propose several enhancement techniques to
strengthen the formulations of the congestion models.
In Section~\ref{sec:modeling-details}, we discuss further practical modeling
techniques and elaborate on our specific choice of congestion measures,
travel cost functions, and uncertainty models used in our computational study.
In Section~\ref{sec:computational-results}, we describe the experimental setup
and the design of the computational study. The results of this study are
discussed in Section~\ref{sec:results}.
Finally, we derive conclusions in Section~\ref{sec:conclusion}.


\section{Problem Statement}
\label{sec:problem-statement}

We study the problem of determining the worst-case congestion
in a traffic network under uncertain travel demand.
To this end, we first introduce the network model and discuss how to account
for congestion and demand uncertainty in Section~\ref{sec:upper-level-problem}.
Afterward, in Section~\ref{sec:lower-level-problem}, we elaborate on the
behavior of the network users, who are assumed to act according to one of the
two Wardrop principles---the user equilibrium or the system optimum.

\subsection{Congestion Model}
\label{sec:upper-level-problem}

To model the multi-commodity traffic network, we consider a directed
graph~$G = (\nodes,\arcs)$ with node set~$\nodes$ and arc set~$\arcs \subseteq
\nodes \times \nodes$.
Moreover, node subsets~$\origs \subseteq \nodes$ and~$\dests \subseteq \nodes$
denote the sets of origin and destination nodes, respectively.
The set of all commodities to be routed through the network is given
by~$\commodities \subseteq \origs \times \dests$ and each
commodity~$k \in \commodities$ has a fixed demand~$d_k \in \R_{\geq 0}$ to
travel from its origin to its destination.
For the ease of presentation, we consider a single commodity for each
origin-destination (OD) pair and all other pairs of nodes are assumed to have
zero demand.
For the remainder of this paper, we make the following connectivity assumption,
which is standard in the transportation literature; see, e.g., Assumption~2.A
in \textcite{Patriksson:2015}.

\begin{assumption}
  \label{as:graph}
  For every commodity~$k = (s_k,t_k) \in \commodities$, there exists at least
  one dipath that connects~$s_k$ and~$t_k$.
\end{assumption}

Let~$\flow = (\flow_a)_{a \in \arcs} \in \R^{\Abs{\arcs}}$ denote the vector of
all arc flows used to model traffic.
To assess the network under adverse conditions, we consider the worst-possible
realization of the travel demand and the corresponding travelers' route
choices. This leads us to considering the problem
\begin{equation}
  \label{eq:upper-level-problem}
  \max_{d,\flow} \quad \latency(\flow)
  \quad \st \quad d \in \uncertaintyset,\, \flow \in F(d).
\end{equation}
Here, we assume that the travel demand~$d = (d_k)_{k \in \commodities}$ is not
known exactly but that it takes values within a given uncertainty
set~$\uncertaintyset$. Moreover, we use~$F(d)$ to denote the set of traffic
flows corresponding to the optimal travelers' responses for a given demand
realization~$d$. For the uncertainty set, we impose the following
throughout the remainder of this paper.

\begin{assumption}
  \label{as:uncertainty-set}
  The set of origin-destination pairs~$\commodities$ is fixed
  across all realizations of the uncertain demand. Moreover,
  the uncertainty set~$\uncertaintyset$ is non-empty, convex, compact, and
  contained in the nonnegative orthant~$\R^{\Abs{\commodities}}_{\geq 0}$.
\end{assumption}

We assume that the uncertainty set~$\uncertaintyset$ is non-empty and compact
to ensure the existence of optimal solutions to the congestion models.
Moreover, we assume convexity to obtain formulations that are more
computationally tractable and amenable to state-of-the-art general-purpose
solvers.
In Assumption~\ref{as:uncertainty-set}, we further require
that~$\uncertaintyset$ is contained in the nonnegative orthant and that the set
of commodities is fixed, which guarantees the existence of optimal flows for
all demand realizations.
Note that, if the travel demand is perfectly known, i.e., if the uncertainty
set~$\uncertaintyset$ is a singleton, Problem~\eqref{eq:upper-level-problem}
reduces to a deterministic traffic assignment problem aimed at determining the
worst-case congestion. Throughout this paper, we refer to
this setting as the \textit{deterministic case} and use it as a baseline for
assessing the impact of demand uncertainty on network congestion.

For the latency function~$\latency(\flow)$ used to measure congestion, we
make the following assumption.

\begin{assumption}
  \label{as:latency-fun}
  The latency function
  $\latency : \R^{\Abs{\arcs}}_{\geq 0} \to \R_{\geq 0},\, \flow \mapsto
  \latency(\flow)$
  is continuous.
\end{assumption}

Next, we elaborate on the Wardrop principles used to model traffic flow.

\subsection{Wardrop Principles}
\label{sec:lower-level-problem}

For computational tractability, we adopt a fully aggregated node-arc
formulation in which all commodities with the same origin are treated
as a unique commodity with a single source and multiple destinations.
In the transportation literature, it is well known that such an aggregation
preserves the set of optimal flows, provided there are no commodity-dependent
costs; see, e.g., Section 5.2.3 in \textcite{Boyles_et_al:2025},
\textcite{Bienstock_et_al:1998}, and \textcite{Chouman_et_al:2017}
in which similar settings are studied.
In what follows, we model aggregated commodity flows using
variables~$\comflow = (\comflow^s)_{s \in \origs}$ with~$\comflow^s =
(\comflow^s_a)_{a \in \arcs} \in \R^{\Abs{\arcs}}_{\geq0}$ for all~$s \in \origs$.
The overall arc flows are then given by
\begin{equation}
  \label{eq:comflow-to-flow}
  \flow = \sum_{s \in \origs} \comflow^s.
\end{equation}
Moreover, for every source~$s \in \origs$, we define the set of commodities
that originate in~$s$~as
\begin{equation*}
  \commodities_s \define \Defset{k = (s_k,t_k) \in \commodities}{s_k = s}.
\end{equation*}
In what follows, we use~$\inarcs{v}$ and~$\outarcs{v}$ to denote the
sets of in- and outgoing arcs of node~\mbox{$v \in \nodes$}, respectively.
Flow conservation can then be stated as
\begin{equation}
  \label{eq:flow-conservation}
  \sum_{a \in \outarcs{v}} \comflow^s_a - \sum_{a \in \inarcs{v}} \comflow^s_a
  = d^s_v, 
  \quad v \in \nodes,\, s \in \origs,
\end{equation}
with
\begin{equation*}
  d^s_v =
  \begin{cases}
    + \sum_{k \in K_s} d_k, & v = s,
    \\
    - d_k,
                            & k = (s,v) \in \commodities_s,
    \\
    0, & \text{otherwise},
  \end{cases}
\end{equation*}
for a given demand realization~$d \in \uncertaintyset$.
With the constraints in~\eqref{eq:comflow-to-flow} and
\eqref{eq:flow-conservation} at hand, which characterize feasible flows in the
traffic network, we now extend the model to incorporate travel costs.
These costs determine how users choose their routes and, consequently, how
flows distribute across the network.
For each arc~$a \in \arcs$, we consider a travel cost function
\begin{equation*}
c_a : \R^{\Abs{\arcs}}_{\geq0} \to \R_{>0}, \quad \flow \mapsto c_a(\flow_a),
\end{equation*}
which captures the time required to traverse the arc as well as additional
expenses such as fuel consumption or congestion-induced delays. In line with
standard terminology, we use the terms travel time and travel cost
interchangeably. Moreover, the following will be a standing assumption for the
remainder of this paper.

\begin{assumption}
  \label{as:cont-cost}
  For every arc~$a \in \arcs$, the function
  $c_a : \R^{\Abs{\arcs}}_{\geq 0} \to \R_{>0},\, \flow \mapsto c_a(\flow_a)$
  is convex, non-decreasing, and continuously differentiable in~$\flow$.
\end{assumption} 

In particular, we assume that travel costs are separable, i.e., the cost of
traversing an arc~$a \in \arcs$ only depends on the flow on that arc but not on
the flows on other arcs of the network.
Assumption~\ref{as:cont-cost} is satisfied by many travel cost functions
studied in the literature such as, e.g., the BPR function; see also Section~1.5
in \textcite{Patriksson:2015} for further discussions.
Let us now formalize the behavior of users in the traffic network.
To this end, we follow the Wardrop principles
\parencite{Wardrop_Whitehead:1952,Wardrop:1952} and distinguish between the
user equilibrium and the system optimum.
For both settings, we assume that travelers have complete knowledge of
available paths and that traffic flows remain stable over time.

\subsubsection*{User Equilibrium}

In the user equilibrium, travelers act selfishly by selecting their routes such
as to minimize their own travel times. This behavior leads to a stable state in
which no commodity can reduce their travel costs by unilaterally changing
routes. Consequently, in a Wardrop user equilibrium, all paths used by a given
commodity have equal travel costs. For separable travel cost functions, the UE
can be obtained as an optimal solution to the problem
\begin{equation}
  \label{eq:user-equi}
  \min_{\flow,\comflow \geq 0} \quad
  \sum_{a \in \arcs} \int_{0}^{\flow_a} c_a(\xi) \diff \xi
  \quad \st \quad
  \eqref{eq:comflow-to-flow},\, \eqref{eq:flow-conservation}.
\end{equation}
Hence, under the user equilibrium, the set~$F(d)$ corresponds to the set of
optimal solutions to Problem~\eqref{eq:user-equi}.
A similar problem is, e.g., considered in Section~3.6.2 in
\textcite{Ferris_Pang:1997}.

\subsubsection*{System Optimum}

The system optimum seeks to minimize the total travel time across
the network by coordinating the flows of all travelers. This corresponds to a
situation in which the marginal travel times are the same on all routes used by
a given commodity. Otherwise, shifting flow from a higher- to a lower-cost
route would reduce the overall travel time. The SO can be obtained as an
optimal solution to the problem
\begin{equation}
  \label{eq:system-opt}
  \min_{\flow,\comflow \geq 0} \quad
  \sum_{a \in \arcs} c_a(\flow_a)\flow_a
  \quad \st \quad
  \eqref{eq:comflow-to-flow},\, \eqref{eq:flow-conservation}.
\end{equation}
Hence, under the system optimum, the set~$F(d)$ corresponds to the set of
optimal solutions to Problem~\eqref{eq:system-opt}.
For further discussion of the SO, see also Section~2.4 in
\textcite{Patriksson:2015}.

To sum up, for each demand realization~$d \in \mathcal{U}$, both
the UE and the SO can be obtained by solving appropriately
chosen nonlinear optimization problems (NLPs). The two models share the same
linear feasibility constraints, i.e., the only nonlinearities arise in the
respective objective functions. As a result, the Abadie constraint
qualification is satisfied for all feasible points. Moreover, under
Assumption~\ref{as:cont-cost}, Problems~\eqref{eq:user-equi}
and~\eqref{eq:system-opt} are convex NLPs, so that the
Karush--Kuhn--Tucker (KKT) conditions are both necessary and sufficient
optimality conditions.
Finally, let us mention that the main difference between the UE and SO models
lies in their behavioral assumptions. The UE reflects decentralized, selfish
route choices by individual travelers, whereas the SO represents a centralized
assignment of traffic flows that minimizes total travel time.

\subsection{Congestion Ratio}
\label{sec:congestion-ratio}

To compare the congestion levels under the user equilibrium and the system
optimum, we introduce the so-called congestion ratio (CR), which we define as
\begin{equation*}
  \text{CR} =
  \frac{\latency(\flow^{\text{UE}})}{\latency(\flow^{\text{SO}})}.
\end{equation*}
Here, $f^{\text{UE}}$ and~$f^{\text{SO}}$ denote optimal arc flows in the
UE and the SO formulation, respectively.
A congestion ratio of~$1$ indicates that the user equilibrium solution achieves
the same overall congestion as the system optimum. Values greater than~$1$
indicate higher congestion under the UE, whereas values lower than~$1$ indicate
higher congestion under the SO.

Note that the congestion ratio differs from the price of anarchy
\parencite{PoA}, a widely used metric in the transportation literature for
comparing UE and SO solutions. The price of anarchy measures the efficiency
loss caused by selfish routing by comparing aggregated travel costs under a
given demand realization and is, by definition, always at least~$1$.
In contrast, the CR focuses on congestion rather than total travel cost and may
compare UE and SO flows corresponding to different demand realizations in our
framework. As a result, the CR captures worst-case congestion levels and,
unlike the price of anarchy, can take values below~$1$.
We illustrate such a behavior with the following example.

\begin{example}[CR $< 1$]
\label{ex:cr-less-than-one}
We consider the network depicted in Figure~\ref{fig:cr-example-network} with a
single commodity~$k = (2,1)$ whose demand is uncertain and takes values in the
interval~$[30,50]$.
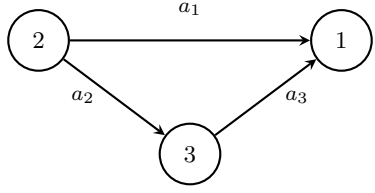
\begin{figure}
\centering
\begin{tikzpicture}[
  transform shape, thick,
  node distance=2.5cm,
  every node/.style={circle, draw, minimum size=0.8cm, font=\small},
  every edge/.style={draw, ->, >=stealth, thick}
  ]
  \node (2) at (0, 0) {$2$};
  \node (1) at (4, 0) {$1$};
  \node (3) at (2, -1.5) {$3$};
  
  \draw[->, >=stealth, thick] (2) -- node[above, draw=none, rectangle,
  font=\footnotesize] {$a_1$} (1); 
  \draw[->, >=stealth, thick] (2) -- node[left, draw=none, rectangle,
  font=\footnotesize] {$a_2$} (3);
  \draw[->, >=stealth, thick] (3) -- node[right, draw=none, rectangle,
  font=\footnotesize] {$a_3$} (1);
    
\end{tikzpicture}
\caption{The traffic network considered in Example~\ref{ex:cr-less-than-one}.}
\label{fig:cr-example-network}
\end{figure}
For all arcs~$a \in \arcs = \set{a_1,a_2,a_3}$, we use the travel cost
function $c_a(\flow_a) = 1 + 0.15 (\flow_a/20)^4$ and measure congestion
as~$\latency(\flow) = \sum_{a \in \arcs} \flow_a/20$, which captures the overall
traffic load.
In this network, there are two paths from the origin node~$2$ to the
destination node~$1$, namely~$2 \to 1$ and~$2 \to 3 \to 1$.

Under the UE formulation, the worst-case demand is~$d^{\text{UE}} = 50$ with
optimal flows~$\flow_{a_1}^{\text{UE}} = 33.2$
and~$\flow_{a_2}^{UE} = \flow_{a_3}^{UE} = 16.8$, which results in a
congestion value of~$\latency(\flow^{\text{UE}}) = 33.2$.
Under the SO formulation, the worst-case demand is still $d^{\text{SO}} =
50$ but the optimal flows are~$\flow_{a_1}^{SO} = 28.4$
and~$\flow_{a_2}^{SO} = \flow_{a_3}^{SO} = 21.6$, which
yields~$\latency(\flow^{\text{SO}}) = 3.58$.
Hence, the congestion ratio is
\begin{equation*}
  \text{CR} = \frac{\latency(\flow^{\text{UE}})}{\latency(\flow^{\text{SO}})} =
  \frac{3.34}{3.58} = 0.93 < 1,
\end{equation*}
which shows that the UE solution can achieve lower worst-case congestion than
the SO solution, a behavior that cannot occur with the price of anarchy.
Note that, in this example, both formulations yield the same
worst-case demand realization, but this does not have to be the case in
general.
\end{example}


\section{Single-Level Reformulations}
\label{sec:reformulations}

Because of its nested structure, the congestion
model~\eqref{eq:upper-level-problem} is intrinsically hard to solve.
Motivated by common approaches in bilevel optimization, we thus
reformulate it as a single-level problem, which can then be
tackled using state-of-the-art general-purpose solvers.
Under Assumption~\ref{as:cont-cost}, the lower-level
problems~\eqref{eq:user-equi} and~\eqref{eq:system-opt} are
convex~$d$-parameterized NLPs, which allows to use their KKT conditions to
obtain equivalent single-level reformulations.
We present the single-level reformulations of the congestion models under the
user equilibrium and the system optimum in
Sections~\ref{sec:single-level-ref:ue} and~\ref{sec:single-level-ref:so},
respectively.
As the resulting problems are mathematical problems with equilibrium
constraints (MPECs)---see, e.g., \textcite{Luo_et_al:1996} for a general
overview---we further reformulate them as mixed-integer nonlinear problems
(MINLPs), following the same ideas as in
\textcite{Fortuny-Amat_Mccarl:1981}. To this end, we introduce auxiliary binary
variables and sufficiently large big-$M$ constants, whose existence we derive
in Section~\ref{sec:big-m}. Finally, we prove the
existence of optimal solutions to the congestion models in
Section~\ref{sec:existence-of-solutions}.

\subsection{Congestion Model  under the User Equilibrium}
\label{sec:single-level-ref:ue}

For a given demand realization~$d \in \uncertaintyset$, the KKT conditions of
Problem~\eqref{eq:user-equi} can be stated as
\begin{subequations}
  \label{eq:user-equi:kkt}
  \begin{align}
    & \flow = \sum_{s \in S} \comflow^s,
      \label{eq:user-equi:comflow-to-flow}
    \\
    & \sum_{a \in \outarcs{v}} \comflow^s_a - \sum_{a \in \inarcs{v}} \comflow^s_a
      = d^s_v, & v \in \nodes,\, s \in \origs,
                 \label{eq:user-equi:flow-cons}
    \\
    & 0 \leq c_a(\flow_a) + \tau_i^s - \tau_j^s \perp \comflow_a^s \geq 0,
    & a = (i,j) \in \arcs,\, s \in \origs.
      \label{eq:user-equi:kkt-complementarity}
  \end{align}
\end{subequations}
Here, the variables~$\tau = (\tau^s)_{s \in \origs}$ with~$\tau^s =
(\tau^s_v)_{v \in \nodes} \in \R^{\Abs{\nodes}}$ correspond to the Lagrangian
multipliers associated with the flow conservation constraints
in~\eqref{eq:flow-conservation}.
Because the KKT conditions are necessary and sufficient
for Problem~\eqref{eq:user-equi}, we can replace Problem~\eqref{eq:user-equi}
with System~\eqref{eq:user-equi:kkt} to obtain the single-level reformulation
of the congestion model~\eqref{eq:upper-level-problem} given by
\begin{equation}
  \label{eq:single-level-ref:ue}
  \max_{d,\flow,\comflow,\tau} \quad \latency(\flow)
  \quad \st \quad
  d \in \uncertaintyset,\, \eqref{eq:user-equi:kkt}.
\end{equation}
Problem \eqref{eq:single-level-ref:ue} is an MPEC because of the complementary
constraints in~\eqref{eq:user-equi:kkt-complementarity}.
Nevertheless, we can exploit the disjunctive nature of these constraints to
derive an MINLP reformulation of the overall congestion model.
To this end, we introduce auxiliary binary variables~$y_a^s \in \set{0,1}$ and
sufficiently large constants~$M_a^s, N_a^s \in \R_{\geq 0}$ for all arcs~$a \in
\arcs$ and all origin nodes~$s \in \origs$ so that we can equivalently re-write
the constraints in~\eqref{eq:user-equi:kkt-complementarity} as
\begin{subequations}
  \label{eq:user-equi:big-m-ref}
  \begin{align}
    0 & \leq c_a(\flow_a) + \tau_i^s - \tau_j^s \leq M_a^s (1 - y_a^s),
    & a = (i,j) \in \arcs,\, s \in \origs,
      \label{eq:user-equi:big-m-ref:cost}
    \\
    0 & \leq \comflow_a^s \leq N_a^s y_a^s,
    & a \in \arcs,\, s \in \origs.
      \label{eq:user-equi:big-m-ref:flow}
  \end{align}
\end{subequations}
For~$a \in \arcs$ and~$s \in \origs$, the binary variable~$y_a^s$
indicates whether any flow originating in~$s$ uses arc~$a$.
The MINLP reformulation of Problem~\eqref{eq:upper-level-problem} then reads
\begin{equation}
  \label{eq:minlp-ref:ue}
  \max_{d,\flow,\comflow,\tau,y} \quad \latency(\flow)
  \quad \st \quad
  d \in \uncertaintyset,\, \text{\eqref{eq:user-equi:comflow-to-flow},
    \eqref{eq:user-equi:flow-cons}, \eqref{eq:user-equi:big-m-ref}},\,
  y^s_a \in \Set{0,1},\, a \in \arcs,\, s \in \origs.
\end{equation}
In general, Problem~\eqref{eq:minlp-ref:ue} is a nonconvex MINLP because
the constraints in~\eqref{eq:user-equi:big-m-ref:cost} are nonlinear and
nonconvex for general convex travel cost functions; cf.\
Assumption~\ref{as:cont-cost}. 
Nevertheless, we emphasize that the structural properties of
Problem~\eqref{eq:minlp-ref:ue} depend strongly on the specific choice of
travel cost and latency functions.
For instance, if both are linear, Problem~\eqref{eq:minlp-ref:ue} is a convex
mixed-integer linear problem (MILP).
We elaborate on problem variants in Section~\ref{sec:modeling-details},
where we study different choices of travel cost and latency functions in more
detail.

\subsection{Congestion Model under the System Optimum}
\label{sec:single-level-ref:so}

For a given demand realization~$d \in \uncertaintyset$, the KKT conditions of
Problem~\eqref{eq:system-opt} can be stated as
\begin{subequations}
  \label{eq:system-opt:kkt}
  \begin{align}
    & \flow = \sum_{s \in S} \comflow^s,
      \label{eq:system-opt:comflow-to-flow}
    \\
    & \sum_{a \in \outarcs{v}} \comflow^s_a - \sum_{a \in \inarcs{v}} \comflow^s_a
      = d^s_v, & v \in \nodes,\, s \in \origs,
                 \label{eq:system-opt:flow-cons}
    \\
    & 0 \leq c_a(\flow_a) + \flow_ac_a'(\flow_a) + \tau_i^s - \tau_j^s \perp
      \comflow_a^s \geq 0,
    & a = (i,j) \in \arcs,\, s \in \origs.
      \label{eq:system-opt:kkt-complementarity}
  \end{align}
\end{subequations}
Here, $c_a'(\flow_a)$ denotes the derivative of the travel cost
function~$c_a(\flow_a)$ for~\mbox{$a \in \arcs$}.
Because the UE formulation~\eqref{eq:user-equi} and the SO
model~\eqref{eq:system-opt} only differ in their objective functions,
the corresponding KKT conditions only differ
in~\eqref{eq:user-equi:kkt-complementarity}
and~\eqref{eq:system-opt:kkt-complementarity}.
As before, we obtain a single-level reformulation of the congestion
model~\eqref{eq:upper-level-problem} by replacing Problem~\eqref{eq:system-opt}
with its necessary and sufficient KKT conditions. This yields
\begin{equation}
  \label{eq:single-level-ref:so}
  \max_{d,\flow,\comflow,\tau} \quad \latency(\flow)
  \quad \st \quad
  d \in \uncertaintyset,\, \eqref{eq:system-opt:kkt}.
\end{equation}
Again, we introduce auxiliary binary variables~$y_a^s \in \set{0,1}$ and
sufficiently large constants~$M_a^s, N_a^s \in \R_{\geq 0}$ for all~$a \in
\arcs$ and~$s \in \origs$ to obtain an equivalent MINLP reformulation
of~\eqref{eq:single-level-ref:so}, which is given by
\begin{subequations}
  \label{eq:minlp-ref:so}
  \begin{align}
    \max_{d,\flow,\comflow,\tau,y} \quad
    & \latency(\flow)
    \\
    \st \quad \
    & d \in \uncertaintyset,\, \text{\eqref{eq:system-opt:comflow-to-flow},
      \eqref{eq:system-opt:flow-cons}},
    \\
    & 0 \leq c_a(\flow_a) + \flow_ac_a'(\flow_a) + \tau_i^s - \tau_j^s
      \leq M^s_a (1 - y^s_a),
    & a \in \arcs,\, s \in \origs,
      \label{eq:minlp-ref:so:big-m-cost}
    \\
    & 0 \leq \comflow_a^s \leq N_a^s y_a^s,
    & a \in \arcs,\, s \in \origs,
    \\
    & y_a^s \in \Set{0,1},
    & a \in \arcs,\, s \in \origs.
  \end{align}
\end{subequations}

\subsection{Variable Bounds and Big-$M$s}
\label{sec:big-m}

Note that the equivalence of Problems~\eqref{eq:minlp-ref:ue}
and~\eqref{eq:minlp-ref:so} to the original congestion models under the UE and
the SO relies on choosing appropriate values for
the constants~$M_a^s, N_a^s$ for all~$a \in \arcs$ and~$s \in \origs$.
In this section, we show that such sufficiently large constants exist by
proving valid bounds for the variables~$\flow$, $\comflow$, and~$\tau$.

\begin{proposition}
  \label{prop:existence-flow-bounds}
  Let~$d \in \uncertaintyset$ be given arbitrarily.
  Then, under Assumptions~\ref{as:graph}, \ref{as:uncertainty-set},
  and~\ref{as:cont-cost}, Problems~\eqref{eq:user-equi}
  and~\eqref{eq:system-opt} each admit an optimal solution~$(\flow, \comflow)$
  that satisfies
  \begin{equation*}
    0 \leq \comflow_a^s \leq \sum_{k \in \commodities_s} d_k,
    \quad a \in \arcs,\, s \in \origs,
  \end{equation*}
  as well as
  \begin{equation*}
    0 \leq \flow_a \leq \sum_{s \in \origs} \sum_{k \in \commodities_s} d_k,
    \quad a \in \arcs.
  \end{equation*}
\end{proposition}
\begin{proof}
  Because the travel cost functions are positive and non-decreasing under
  Assumption~\ref{as:cont-cost}, any positive flow on a cycle would increase
  the objective function value and therefore cannot occur in an optimal
  solution; cf., e.g., Theorem~3.8 in \textcite{Ahuja_et_al:1993}.
  Hence, we can, w.l.o.g., add the
  constraints~$0 \leq \comflow_a^s \leq \sum_{k \in \commodities_s} d_k$ for
  all~$a \in \arcs$ and~$s \in \origs$ to Problems~\eqref{eq:user-equi}
  and~\eqref{eq:system-opt} without affecting their sets of optimal solutions.
  As the variables~$\comflow$ are linearly coupled to the arc flows~$\flow$
  via~\eqref{eq:comflow-to-flow}, this directly implies finite bounds
  for~$\flow$, i.e., all flow variables are bounded.
  The feasibility of Problems~\eqref{eq:user-equi}
  and~\eqref{eq:system-opt} follows from Assumptions~\ref{as:graph}
  and~\ref{as:uncertainty-set}.
  Moreover, as the feasible sets are described by finitely many linear
  constraints, they are compact.
  Finally, because the objective functions of the UE and SO problems are
  continuous, the Weierstrass theorem yields the existence of optimal solutions
  to Problems~\eqref{eq:user-equi} and~\eqref{eq:system-opt}; see, e.g.,
  Theorem~2.4 in \textcite{Patriksson:2015} for similar arguments.
\end{proof}

\begin{proposition}
  \label{prop:bound-tau}
  Let~$d \in \uncertaintyset$ be given arbitrarily and suppose
  that Assumptions~\ref{as:graph}, \ref{as:uncertainty-set},
  and~\ref{as:cont-cost} hold. Then, for
  every~$(\flow, \comflow) \in F(d)$, there exists~$\tau$ such
  that~$0 \leq \tau_v^s < \infty$ holds for all~$v \in \nodes$
  and~$s \in \origs$, and~$(\flow,\comflow,\tau)$
  solves~\eqref{eq:user-equi:kkt} or~\eqref{eq:system-opt:kkt} under the user
  equilibrium or the system optimum, respectively.
\end{proposition}
\begin{proof}
  The claim can be shown in analogy to Remark~1 and the proof of Proposition~3
  in \textcite{Beck_et_al:2024}.
\end{proof}

Finally, we mention that sufficiently large constants~$M^s_a$ and~$N^s_a$,
$a \in \arcs$, $s \in \origs$, to be used in Problems~\eqref{eq:minlp-ref:ue}
and~\eqref{eq:minlp-ref:so} can be obtained by exploiting
Propositions~\ref{prop:existence-flow-bounds} and~\ref{prop:bound-tau} together
with Assumption~\ref{as:cont-cost} and the specific structure of the travel
cost functions; see, e.g., \textcite{Beck_et_al:2024} for further details.

\subsection{Existence of Solutions}
\label{sec:existence-of-solutions}

To conclude this section, we now prove the existence of optimal solutions to
the congestion model~\eqref{eq:upper-level-problem} under the UE and the SO.

\begin{theorem}
  \label{thm:existence-of-sols}
  Under Assumptions~\ref{as:graph}--\ref{as:cont-cost}, the congestion
  model~\eqref{eq:upper-level-problem} admits an optimal solution under
  both the user equilibrium and the system optimum.
\end{theorem}
\begin{proof}
  The function~$\latency$ is continuous because of
  Assumption~\ref{as:latency-fun}. Moreover, the set~$\uncertaintyset$ is
  non-empty and compact due to Assumption~\ref{as:uncertainty-set} and, by
  Propositions~\ref{prop:existence-flow-bounds} and~\ref{prop:bound-tau}, the
  set~$F(d)$ is non-empty and bounded for all~$d \in \uncertaintyset$.
  In addition, the graph of the set-valued map~$F$ is closed as it is described
  by finitely many continuous equality and inequality constraints.
  Applying the Weierstra\ss\ theorem completes the proof.
\end{proof}


\section{Strengthened Formulations}
\label{sec:tightening}

We now describe the main strategies used to improve the computational
tractability of the congestion model~\eqref{eq:upper-level-problem}.
These techniques reduce the problem size and strengthen the formulation by
exploiting the network topology and structural properties of optimal flows.
In Section~\ref{sec:partitioning}, we discuss graph partitioning
techniques based on cut nodes and blocks that can be used to derive tighter
bounds for the flow variables~$\flow$ and~$\comflow$, which is what we do
in Section~\ref{sec:bound-tightening}.
In Section~\ref{sec:no-cycles}, we present valid inequalities that
eliminate cycle flows in the traffic network.

\subsection{Graph Partitioning}
\label{sec:partitioning}

Given the directed graph~$G = (\nodes, \arcs)$, we denote its
underlying simple undirected graph as~$\underline{G} = (\nodes, E)$ in which 
each arc $(i,j)$ is replaced by edge $\{i,j\}$ and no multiple edges are
allowed.
Before proceeding, we introduce some basic graph-theoretic concepts that will
be used throughout this section; see, e.g.,
Chapter~4 in \textcite{West:2001} for further details.

\begin{definition}[Connected Component]
  \label{def:connected-component}
  A \emph{connected component} of~$\underline{G}$ is a maximal
  subset~$C \subseteq \nodes$ such that, for all~$u,v \in C$,
  there exists a path connecting~$u$ and~$v$.
\end{definition}

\begin{definition}[Articulation Node]
  \label{def:articulation-node}
  An \emph{articulation node} (or \emph{cut node}) is a node whose removal
  increases the number of connected components of~$\underline{G}$.
\end{definition}

\begin{definition}[Bridge]
  \label{def:bridge}
  An edge~$e \in E$ is a \emph{bridge} (or \emph{cut edge}) if the
  graph \mbox{$(\nodes, E \setminus \set{e})$} has more connected components
  than~$\underline{G}$.
\end{definition}

\begin{definition}[$2$-Connected Graph]
  \label{def:2-connected}
  An undirected graph~$\underline{G}$ is said to be \emph{$2$-connected} (or
  \emph{node-biconnected}) if it is connected, has at least three nodes,
  and contains no cut node.
\end{definition}

\begin{definition}[Block]
  \label{def:block}
  A \emph{block} $B=(\nodes_B, \edges_B)$ of~$\underline{G}$ is a maximal
  connected subgraph of~$\underline{G}$ that has no cut node.
  If a block has a single edge, it is called a \emph{trivial block} or
  \emph{bridge}, otherwise it is called a \emph{non-trivial block}.
\end{definition}

Note that~$\underline{G}$ itself is a block if it is connected and has no cut
node. Moreover, we emphasize that every edge of~$\underline{G}$ belongs to
exactly one block and that two blocks share at most one node, which must be
an articulation node.

\begin{definition}[Block-Cut Tree]
  \label{def:block-cut-tree}
  The \emph{block-cut tree} of a connected graph~$\underline{G}$ is a
  tree~$T$ whose node set consists of the blocks and articulation nodes
  of~$\underline{G}$, with an edge between a block~$B$ and an articulation
  node~$v$ whenever~$v \in B$.
\end{definition}

For a set of nodes~$S \subseteq \nodes$, we denote by~$\underline{G}[S]$
and~$G[S]$ the subgraphs of~$\underline{G}$ and~$G$ induced by~$S$,
respectively.
In the following, we consider graph decomposition techniques based on blocks
and articulation nodes. Such decompositions are not only theoretically
appealing but also relevant for real-world networks. Empirical evidence
indicates that many practical networks contain a non-negligible fraction of
articulation points, e.g., \textcite{Tian_et_al:2017} document
substantial articulation structure in U.S.\ road networks, underscoring the
practical relevance of such decomposition techniques.

\begin{example}
  \label{ex:articulation-nodes}
  We consider the undirected graph~$\underline{G} = (\nodes, E)$
  with node set~$\nodes = \set{1,\ldots,11}$ and edge set~$E =
  \set{e_1,\ldots,e_{14}}$ illustrated in Figure~\ref{fig:graph_an}~(a).
  This graph has three articulation nodes, $c_1 = 1$, $c_2 = 5$, and~$c_3 = 9$,
  and six blocks:
  \begin{align*}
    B_1 & = (\set{1,2,3},\set{e_1,e_2,e_3}),\
          B_2 = (\set{1,4,5},\set{e_4,e_5,e_6}),
    \\
    B_3 & = (\set{5,6,7},\set{e_7,e_8,e_9}),\,
          B_4 = (\set{5,8,9},\set{e_{10},e_{11},e_{12}}),
    \\
    B_5 & = (\set{9,10},\set{e_{13}}),\,
          B_6 = (\set{9,11},\set{e_{14}}).
  \end{align*}
  Here, $B_1$, $B_2$, $B_3$, and~$B_4$ are non-trivial blocks ($2$-connected
  subgraphs), whereas~$B_5$ and~$B_6$ are trivial blocks (bridges).
  The corresponding block-cut tree is shown in
  Figure~\ref{fig:graph_an}~(b).
  \begin{figure}
  \centering
  \begin{tikzpicture}[
    transform shape, thick,
    every node/.style={font=\small},
    vertex/.style={circle, draw=black, fill=white, text=black, minimum
      size=0.45cm, font=\footnotesize\sffamily, inner sep=1pt},
    artnode/.style={circle, draw=black, fill=black, text=white, minimum
      size=0.45cm, font=\footnotesize\sffamily, inner sep=1pt},
    edge/.style={draw=black, line width=0.8pt},
    elabel/.style={font=\small, inner sep=1.5pt}
    ]
    
    \node[artnode] (1) at (1.5, 4) {$1$};
    \node[vertex] (2) at (0, 5.0) {$2$};
    \node[vertex] (3) at (0, 3) {$3$};
    \node[vertex] (4) at (3.5, 3) {$4$};
    \node[artnode] (5) at (3.5, 4.5) {$5$};
    \node[vertex] (6) at (1.5, 5.0) {$6$};
    \node[vertex] (7) at (3.5, 5.5) {$7$};
    \node[vertex] (8) at (5.5, 3) {$8$};
    \node[artnode] (9) at (7, 4.5) {$9$};
    \node[vertex] (10) at (7, 3.5) {$10$};
    \node[vertex] (11) at (8.5, 4.5) {$11$};
    
    \draw[edge] (1) -- (2) node[elabel, pos=0.5, above right] {$e_1$};
    \draw[edge] (1) -- (3) node[elabel, pos=0.5, below right] {$e_2$};
    \draw[edge] (3) -- (2) node[elabel, pos=0.5, below left] {$e_3$};
    \draw[edge] (1) -- (4) node[elabel, pos=0.5, below left] {$e_4$};
    \draw[edge] (1) -- (5) node[elabel, pos=0.5, above left] {$e_5$};
    \draw[edge] (4) -- (5) node[elabel, pos=0.4, right] {$e_6$};
    \draw[edge] (5) -- (6) node[elabel, pos=0.5, above right] {$e_7$};
    \draw[edge] (6) -- (7) node[elabel, pos=0.5, above] {$e_8$};
    \draw[edge] (7) -- (5) node[elabel, pos=0.5, right] {$e_{9}$};
    \draw[edge] (5) -- (9) node[elabel, pos=0.5, above] {$e_{10}$};
    \draw[edge] (5) -- (8) node[elabel, pos=0.5, above right] {$e_{11}$};
    \draw[edge] (9) -- (8) node[elabel, pos=0.5, above left] {$e_{12}$};
    \draw[edge] (9) -- (10) node[elabel, pos=0.5, right] {$e_{13}$};
    \draw[edge] (9) -- (11) node[elabel, pos=0.5, above] {$e_{14}$};
    \node[anchor=north, font=\small] at (3.75, 2) {(a) Graph~$\underline{G}$};
  \end{tikzpicture}
  
  \bigskip
  
  \begin{tikzpicture}[
    scale=0.75, transform shape, thick,
    every node/.style={font=\small},
    treenode/.style={circle, fill=black, minimum size=5pt, inner sep=0pt},
    edge/.style={draw=black, line width=0.8pt},
    nlabel/.style={font=\footnotesize\sffamily}
    ]
    
    
    \node[treenode, label={below:{$B_1$}}] (B1) at (0,   2) {};
    \node[treenode, label={below:{$c_1$}}] (c1) at (1.5, 2) {};
    \node[treenode, label={below:{$B_2$}}] (B2) at (3,   2) {};
    \node[treenode, label={below:{$c_2$}}] (c2) at (4.5, 2) {};
    \node[treenode, label={right:{$B_3$}}] (B3) at (4.5, 3) {};
    \node[treenode, label={below:{$B_4$}}] (B4) at (6,   2) {};
    \node[treenode, label={above:{$c_3$}}] (c3) at (7.5, 2) {};
    \node[treenode, label={right:{$B_5$}}] (B5) at (7.5, 1) {};
    \node[treenode, label={below:{$B_6$}}] (B6) at (9,   2) {};
    
    \draw[edge] (B1) -- (c1);
    \draw[edge] (c1) -- (B2);
    \draw[edge] (B2) -- (c2);
    \draw[edge] (c2) -- (B3);
    \draw[edge] (c2) -- (B4);
    \draw[edge] (B4) -- (c3);
    \draw[edge] (c3) -- (B5);
    \draw[edge] (c3) -- (B6);
    
    \node[anchor=north, font=\small] at (4.5, 0.8) {(b) Block-cut tree~$T$};
  \end{tikzpicture}
  \caption{The graph considered in Example~\ref{ex:articulation-nodes} with
    articulation nodes shown in black (a) and its block-cut tree (b).}%
  \label{fig:graph_an}%
\end{figure}
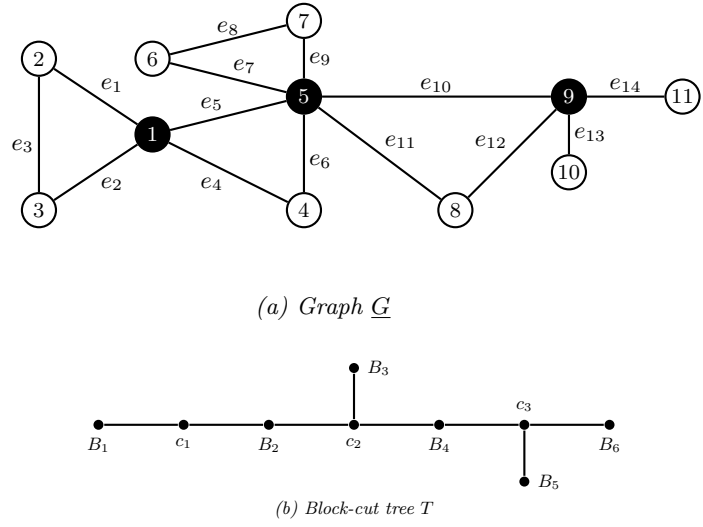%

\end{example}

Under Assumption~\ref{as:cont-cost}, UE and SO solutions
cannot have positive flow on directed cycles; see, e.g., Theorem~3.8 in
\textcite{Ahuja_et_al:1993}.
This property also holds for the source-aggregated formulation
in Problems~\eqref{eq:user-equi} and~\eqref{eq:system-opt}; see,
e.g., Section~5.2.3 in \textcite{Boyles_et_al:2025} for further details.
The absence of directed cycle flows allows us to strengthen the formulation of
the congestion models by fixing the values of certain variables in a
preprocessing step.
Specifically, any dipath that exits a block through an articulation node and
re-enters it through the same node forms a directed cycle. Hence,
all optimal flows with origin and destination nodes in the same block can
be restricted to dipaths entirely contained in that block.

\begin{proposition}
  \label{prop:partitions-paths}
  Suppose that Assumption~\ref{as:cont-cost} holds.
  Let~$B = (\nodes_B, \edges_B) \subseteq \underline{G}$ be a block and
  let $u, v \in \nodes_B$ be such that there exists a dipath from~$u$ to~$v$
  in~$G$. Then, the following statements~hold.
  \begin{enumerate}[label=(\roman*)]
  \item There exists at least one dipath from~$u$ to~$v$ in~$G$
    that is entirely contained in~$G[\nodes_B]$.
  \item Any dipath from~$u$ to~$v$ used in an optimal UE or SO solution is
    entirely contained in~$G[\nodes_B]$.
  \end{enumerate}
\end{proposition}
\begin{proof}
  We first recall a key property of block-cut trees that will be used in the
  proof. To this end, let~$T$ be a block-cut tree of~$\underline{G}$, i.e., a
  block~$B$ is a node of~$T$ and its neighbors in~$T$ are precisely the cut
  nodes of~$\underline{G}$ contained in~$B$.
  If a path in~$\underline{G}$ leaves~$B$ through a node~$w \in \nodes_B$,
  then~$w$ must be a cut node. This is because the edge connecting~$w$ to a
  node outside of~$\nodes_B$ belongs to a different block~$B'$ and~$w \in B
  \cap B'$ implies that~$w$ is a cut node.
  Once the path leaves~$B$ through~$w$, it continues within the subtree of~$T -
  B$ that contains~$w$. Because~$T$ is a tree, the only way for the path to
  return to~$B$ is through node~$w$.
  Because every arc in~$G$ corresponds to an edge in~$\underline{G}$, the same
  property holds for dipaths in~$G$: any dipath that exits~$G[\nodes_B]$ must
  re-enter through the same cut node~$w$, thereby forming a directed cycle
  through~$w$.
 
  We now prove the two statements. By assumption, there exists a dipath
  from~$u$ to~$v$ in~$G$. As shown above, each time this dipath
  leaves~$G[\nodes_B]$, it must re-enter through the same cut node, forming a
  directed cycle. Removing all such cycles yields a dipath from~$u$ to~$v$ that
  is entirely contained in~$G[\nodes_B]$, which proves~(i).
  Under Assumption~\ref{as:cont-cost}, optimal UE and SO solutions do not carry
  flow on directed cycles.
  Due to~(i), any dipath that leaves and later re-enters~$G[\nodes_B]$
  forms a directed cycle.
  Hence, all dipaths carrying positive flow in an optimal solution must be
  entirely contained in~$G[\nodes_B]$, which proves~(ii).
\end{proof}

By Part~(ii) of Proposition~\ref{prop:partitions-paths}, we can a priori fix
certain variables of the congestion model~\eqref{eq:upper-level-problem} to
zero. More formally, we have the following result.

\begin{corollary}
  \label{cor:fix-vars-outside-blocks-path}
  Suppose that Assumptions~\ref{as:graph}--\ref{as:cont-cost} hold.
  Let~$(d,\flow,\comflow,\tau,y)$ be an optimal solution to either
  Problem~\eqref{eq:minlp-ref:ue} or Problem~\eqref{eq:minlp-ref:so}.
  Further, let~$s \in \origs$ be given arbitrarily and let~$B_s$ denote
  a block of~$\underline{G}$ containing~$s$.
  For each~$k = (s, t_k) \in \commodities_s$, let~$B_{t_k}$ be a block
  of~$\underline{G}$ containing~$t_k$ and let~$\mathcal{B}_k$ be the
  set of blocks on the unique path from~$B_s$ to~$B_{t_k}$ in the
  block-cut tree~$T$ of~$\underline{G}$.
  Moreover, let
  \begin{equation*}
    \mathcal{B}_s = \bigcup_{k \in \commodities_s} \mathcal{B}_k
    \quad \text{and} \quad
    \arcs_{\mathcal{B}_s} = \Defset{(i,j) \in \arcs}{\set{i,j}
      \in E_{B} \text{ for some } B \in \mathcal{B}_s}.
  \end{equation*}
  Then, it holds
  \begin{equation*}
    \comflow^s_a = 0,\, y^s_a = 0, \quad a \notin \arcs_{\mathcal{B}_s}.
  \end{equation*}
\end{corollary}

We can also fix additional variables that would take the value of
zero in any optimal solution. For instance, we can a priori fix the
variables~$\comflow^s_a$ and~$y^s_a$ for all arcs~$a \in \arcs$ that
are not reachable from an origin~$s \in \origs$.

\begin{lemma}
  \label{lem:fix-vars-unreachable}
  Let~$d \in \uncertaintyset$ be given arbitrarily.
  For each~$s \in \origs$, let the set of nodes that are reachable from
  source~$s$ be given by
  \begin{equation*}
    \nodes_s = \Defset{v \in \nodes}{\text{there exists a $s$--$v$ dipath}}.
  \end{equation*}
  Then, for any~$(\flow,\comflow) \geq 0$ that
  satisfies~\eqref{eq:comflow-to-flow} and~\eqref{eq:flow-conservation}, it
  holds~$\comflow^s_a = 0$ for all~$s \in \origs$ and~$a = (i,j) \in \arcs$
  with~$i \notin \nodes_s$.
\end{lemma}
\begin{proof}
  The claim immediately follows from the definition of the set~$\nodes_s$, $s
  \in \origs$.
\end{proof}

We use Lemma~\ref{lem:fix-vars-unreachable} to fix some of the
binary variables~$y^s_a$ in Problems~\eqref{eq:minlp-ref:ue}
and~\eqref{eq:minlp-ref:so}, which indicate whether arc~$a \in \arcs$ is
used by any commodity with origin~$s \in \origs$.

\begin{corollary}
  \label{cor:fix-vars-unreachable}
  Suppose that Assumptions~\ref{as:graph}--\ref{as:cont-cost} hold.
  Let~$(d,\flow,\comflow,\tau,y)$ be feasible for either
  Problem~\eqref{eq:minlp-ref:ue} or Problem~\eqref{eq:minlp-ref:so}.
  For each~$s \in \origs$, let the set of nodes~$\nodes_s$ reachable
  from source~$s$ be defined as in Lemma~\ref{lem:fix-vars-unreachable}.
  Then, $y^s_a = 0$ holds for all~$s \in \origs$ and~$a = (i,j) \in \arcs$
  with~$i \notin \nodes_s$.
\end{corollary}

Lemma~\ref{lem:fix-vars-unreachable} also allows us to reduce the number of
constraints in the model by removing the redundant flow conservation
constraints
\begin{equation*}
  \sum_{a \in \outarcs{v}} \comflow^s_a - \sum_{a \in \inarcs{v}} \comflow^s_a
  = 0, 
  \quad s \in \origs,\, v \notin \nodes_s,
\end{equation*}
from the original congestion model.

\subsection{Improved Variable Bounds}
\label{sec:bound-tightening}

Building on the results of Section~\ref{sec:partitioning}, we now
compute tighter bounds for the flow variables~$\flow$ and~$\comflow$,
which can be used to obtain smaller values for the big-$M$
constants in Problems~\eqref{eq:minlp-ref:ue} and~\eqref{eq:minlp-ref:so}.

\begin{proposition}
  \label{prop:tightened-flow-bounds}
  Suppose that Assumptions~\ref{as:graph}, \ref{as:uncertainty-set},
  and~\ref{as:cont-cost} hold, and
  let~$d \in \uncertaintyset$ and $(\flow,\comflow) \in F(d)$ be given
  arbitrarily.
  For each~$s \in \origs$ and~$k \in \commodities_s$,
  let~$\mathcal{B}_k$ be defined as in
  Corollary~\ref{cor:fix-vars-outside-blocks-path} and
  \begin{equation*}
    \arcs_{\mathcal{B}_k} \define \Defset{(i,j) \in \arcs}{\set{i,j} \in E_B
    \text{ for some } B \in \mathcal{B}_k}.
  \end{equation*}
  Moreover, let the set of nodes~$\nodes_s$ reachable from source~$s$ be
  defined as in Lemma~\ref{lem:fix-vars-unreachable}.
  Then, for all~$a = (i,j) \in \arcs$ and~$s \in \origs$, the flows satisfy
  \begin{equation*}
    0 \leq \comflow_a^s \leq \bar \comflow_a^s \define
    \sum_{k \in \commodities_s} d_k \vv{1}_{\set{a \in \arcs_{\mathcal{B}_k}}},
  \end{equation*} 
  as well as
  \begin{equation*}
    0 \leq \flow_a \leq \bar \flow_a \define
    \sum_{s \in \origs} \sum_{k \in \commodities_s}
    d_k \vv{1}_{\set{a \in \arcs_{\mathcal{B}_k}}}, 
  \end{equation*} 
  where~$\vv{1}_{\set{a \in \arcs_{\mathcal{B}_k}}}$ is an indicator function
  that equals~$1$ if~$a \in \arcs_{\mathcal{B}_k}$ and~$0$ otherwise.
\end{proposition}
\begin{proof}
  The claim immediately follows from
  Proposition~\ref{prop:existence-flow-bounds},
  Corollary~\ref{cor:fix-vars-outside-blocks-path},
  Lemma~\ref{lem:fix-vars-unreachable}, and the
  constraints in~\eqref{eq:comflow-to-flow}.
\end{proof}

\subsection{Cycle-Elimination Constraints}
\label{sec:no-cycles}

Because UE and SO solutions cannot carry positive flow on cycles, we can
leverage this property to further strengthen the MINLP reformulations of the
congestion models by adding cycle-elimination constraints.
To this end, let~$\mathcal{R}$ denote the set of all simple directed cycles in
the network. Then, for all~$r \in \mathcal{R}$ and~$s \in \origs$, the
associated cycle-elimination constraint is given by
\begin{equation}
  \label{eq:cycle-constraint}
  \sum_{a \in r} y_a^s \leq \Abs{r} - 1.
\end{equation}
Here, $y_a^s$ is the binary variable indicating whether arc~$a \in r$ is
used by any commodity originating in~$s \in \origs$ and~$\Abs{r}$ is the
number of arcs in (or the length of) the cycle~$r \in \mathcal{R}$.
Constraint~\eqref{eq:cycle-constraint} ensures that at
most~$\Abs{r} - 1$ arcs can carry positive flow, thereby preventing any
directed cycles in an optimal solution.
To reduce the number of cycle-elimination constraints added to the model, we
exploit Corollary~\ref{cor:fix-vars-outside-blocks-path} so that
Constraint~\eqref{eq:cycle-constraint} only needs to be imposed for cycles
contained within the blocks associated with a source node~$s$.

\begin{proposition}
  Suppose that Assumptions~\ref{as:graph}--\ref{as:cont-cost} hold.
  Let~$s \in \origs$ be given arbitrarily and let the set~$\mathcal{B}_s$ be
  the set of blocks associated with~$s$ as defined in
  Corollary~\ref{cor:fix-vars-outside-blocks-path}.
  For each~$B \in \mathcal{B}_s$, let~$\nodes_B$ denote the set of nodes in
  block~$B$. Further, let
  \begin{equation}
    \label{eq:cycle-set}
    \mathcal{R}_s \define \Defset{r \in \mathcal{R}}{\nodes(r) \subseteq
      \nodes_B \text{ for some } B \in \mathcal{B}_s},
  \end{equation}
  where~$\nodes(r)$ denotes the set of nodes visited by
  cycle~$r \in \mathcal{R}$. Then, the inequalities
  \begin{equation}
    \label{eq:cycle-constraint-partitioned}
    \sum_{a \in r} y_a^s \leq \Abs{r} - 1,
    \quad r \in \mathcal{R}_s
  \end{equation}
  are valid for Problems~\eqref{eq:minlp-ref:ue} and~\eqref{eq:minlp-ref:so}.
\end{proposition}
\begin{proof}
  We prove the claim for the UE as the SO case can be shown analogously.
  Let~$(d,\flow,\comflow,\tau,y)$ be feasible for
  Problem~\eqref{eq:minlp-ref:ue}. We prove the claim by contradiction.
  To this end, suppose that there exist~$s \in \origs$ and~$r \in
  \mathcal{R}_s$ with
  \begin{equation*}
    \sum_{a \in r} y_a^s > \Abs{r} - 1.
  \end{equation*}
  Because~$y^s \in \set{0,1}^{\Abs{\arcs}}$, this implies~$y_a^s = 1$ for
  all~$a \in r$. By Constraints~\eqref{eq:user-equi:big-m-ref:cost}, we then
  have~$c_a(\flow_a) + \tau_i^s - \tau_j^s = 0$ for all~$a = (i,j) \in r$.
  Summing these equalities over all arcs~$a = (i,j) \in r$ yields
  \begin{equation*}
    \sum_{a \in r} c_a(\flow_a) = 0,
  \end{equation*}
  which contradicts Assumption~\ref{as:cont-cost}. This concludes the proof.
\end{proof}


\section{Congestion Measures, Travel Cost Functions\\and
Demand Uncertainty Modeling}
\label{sec:modeling-details}

In the previous sections, we have presented a general framework for modeling and
solving the congestion model under demand uncertainty, which is applicable to a
wide range of congestion measures, travel cost functions, and uncertainty sets.
In this section, we now specify the modeling choices that we will adopt in our
computational study in Sections~\ref{sec:computational-results}
and~\ref{sec:results}.
Specifically, we consider three congestion measures
(Section~\ref{sec:latency-functions}), three travel cost functions
(Section~\ref{sec:travel-costs}), and three uncertainty sets
(Section~\ref{sec:uncertainty-set}), which are commonly used
modeling choices in the transportation literature. We also discuss practical
implementation details if relevant.

Throughout this section, we use $0 < \capacity_a < \infty$ to denote the
practical capacity of an arc~$a \in \arcs$, which represents the maximum flow
that an arc can accommodate under ideal operating conditions without causing
congestion; see, e.g., Section~1.5 in \textcite{Patriksson:2015} for further
details.

\subsection{Latency Functions}
\label{sec:latency-functions}

We measure congestion using one of the following three latency functions.
We emphasize that all three considered congestion measures are continuous in
the flows~$\flow$, i.e., they satisfy Assumption~\ref{as:latency-fun}.

\subsubsection*{Maximum Utilization Function}
We measure congestion by focusing on the most heavily used arc in the
network. To this end, we consider the function
\begin{equation}
  \label{eq:max-ratio-fun}
  \latency(\flow) = \max_{a \in \arcs} \, \frac{\flow_a}{\capacity_a}.
\end{equation}
In particular, \eqref{eq:max-ratio-fun} can be used to identify the network's
most severe bottleneck.

\subsubsection*{Total Utilization Function}
Whereas the previous function focuses on bottlenecks in the network, we
also study the overall traffic load in the network by considering the linear
function
\begin{equation}
  \label{eq:sum-ratio-fun}
  \latency(\flow) = \sum_{a \in \arcs} \frac{\flow_a}{\capacity_a}.
\end{equation}

\subsubsection*{Bureau of Public Roads (BPR) Function}
Another congestion measure that is widely used in the literature is the
BPR function \parencite{Bureau-Public-Roads:1964}, which is given by
\begin{equation}
  \label{eq:bpr-function-latency}
  \latency(\flow) = \sum_{a \in \arcs} \bpr(\flow_a)
\end{equation}
with
\begin{equation}
  \label{eq:bpr-function}
  \bpr(\flow_a) = c^{\fix}_a \left( 1 + \alpha {\left(
        \frac{\flow_a}{\capacity_a} \right)}^\beta \right),
  \quad a \in \arcs.
\end{equation}
Here, $c^{\fix}_a \in \R_{> 0}$ is the free-flow time of arc~$a \in
\arcs$, i.e., the travel time experienced if the arc is used under ideal
operating conditions.
Moreover, the parameters~\mbox{$\alpha, \beta \in \R_{> 0}$} can be adjusted to
account for how quickly travel times increase as flows approach the practical
capacity of an arc. Commonly used values are~$\alpha=0.15$ and~$\beta=4$, which
reflect typical empirical observations of congestion effects and will also be
used in our computational study in Sections~\ref{sec:computational-results}
and~\ref{sec:results}.
Note that the BPR function~\eqref{eq:bpr-function} is convex for~$\flow \geq 0$
in this case.

We point out that the maximum utilization function~\eqref{eq:max-ratio-fun}
and degree-$4$ polynomials such as those in the BPR function are not directly
supported by general-purpose solvers such as \textsf{Gurobi}. To handle
these cases, we thus apply the following reformulation techniques.

\subsubsection{Reformulation of the Maximum Utilization Function}

When measuring congestion using the maximum utilization function, we consider
the problem
\begin{equation}
  \label{eq:congestion-w-max-ratio}
  \max_{d,\flow} \ \max_{a \in \arcs} \,
  \frac{\flow_a}{\capacity_a}
  \quad \st \quad d \in \uncertaintyset,\, \flow \in F(d).
\end{equation}
Because the outer maximization corresponds to an optimization problem and the
inner maximization is taken over a discrete set of arcs, the resulting
max–max structure cannot be reformulated using a standard epigraph
reformulation.
Nevertheless, an equivalent MINLP reformulation of
Problem~\eqref{eq:congestion-w-max-ratio} can be obtained by introducing
auxiliary binary variables~\mbox{$w_a \in \set{0,1}$} for all~$a \in \arcs$ as
well as a sufficiently large constant~$M \in \R_{>0}$.
Specifically, Problem~\eqref{eq:congestion-w-max-ratio} is equivalent to
\begin{subequations}
  \label{eq:epi-ref}
  \begin{align}
    \max_{d,\flow,\comflow,\lambda} \quad
    & \lambda
    \\
    \st \quad
    & \lambda \leq \frac{f_a}{\capacity_a} + M(1 - w_a), \quad a \in \arcs,
      \label{eq:epi-ref:epi}
    \\ 
    & \sum_{a \in \arcs} w_a = 1, \\
    & w_a \in \Set{0,1}, \quad a \in \arcs, \\
    & d \in \uncertaintyset,\, (\flow,\comflow) \in F(d).
  \end{align}
\end{subequations}
In Problem~\eqref{eq:epi-ref}, exactly one of the~$w$ variables takes the
value~$1$. If~$w_a = 1$ holds for some arc~$a \in \arcs$,
the corresponding constraint in~\eqref{eq:epi-ref:epi}
enforces~$\lambda \leq \flow_a/\capacity_a$.
Because the objective is to maximize~$\lambda$, this inequality is satisfied
with equality for the arc attaining the maximum utilization ratio.
All other~$w$ variables take the value~$0$.
Nevertheless, to ensure the correctness of the formulation,
it is essential to select a sufficiently large constant~$M$.
A valid choice is
\begin{equation}
  M \define \max_{a \in \arcs} \Set{\frac{\bar \flow_a}{\capacity_a}},
\end{equation}
where~$\bar \flow_a$ denotes an upper bound on the arc flows; see
Proposition~\ref{prop:tightened-flow-bounds}.

\subsubsection{Handling Degree-$4$ Polynomials}
\label{sec:handling-degree-4}

Nowadays, general-purpose MILP solvers such as
\textsf{Gurobi} can handle convex quadratic programs (MIQPs) but
higher-degree polynomials are not directly supported.
To address the congestion model with the BPR latency function, we thus
introduce auxiliary
variables \mbox{$\phi \in \R^{\Abs{\arcs}}_{\geq 0}$} modeling
\begin{equation*}
  \phi_a = f_a \cdot f_a, \quad a \in \arcs,
\end{equation*}
so that degree-$4$ polynomials are expressed using only quadratic terms.

\subsection{Travel Cost Functions}
\label{sec:travel-costs}

We also use the BPR function presented in~\eqref{eq:bpr-function-latency} for
the travel costs. In this paper, we focus on settings with~$\alpha = 0.15$
and \mbox{$\beta \in \set{1,2,4}$}.
For~$\beta = 1$, the cost functions are linear; for~$\beta = 2$, they are
quadratic; and for~$\beta = 4$, we obtain the classic BPR function that is
commonly studied in the literature and also used as the latency function in
this paper.
In all three cases, the travel cost functions are convex, non-decreasing, and
continuously differentiable in~$f \in \R^{\Abs{\arcs}}_{\geq 0}$, i.e., they
satisfy Assumption~\ref{as:cont-cost}.
To handle the degree-$4$ polynomials in the travel cost functions
with~$\beta = 4$, we apply the reformulation described in
Section~\ref{sec:handling-degree-4}.

\begin{remark}
  \label{rem:structural-properties}
  In Problem~\eqref{eq:minlp-ref:so}, the derivatives~$c'_a(\flow_a)$ are
  polynomials of one degree lower than the original travel cost
  functions~$c_a(\flow_a)$. Hence, the terms~$\flow_a c'_a(\flow_a)$,
  \mbox{$a \in \arcs$}, in~\eqref{eq:minlp-ref:so:big-m-cost} are polynomials
  of the same degree as~$c_a(\flow_a)$.
  As a result, considering the SO instead of the UE does not affect the
  structural properties of the MINLP reformulation of the congestion model.
\end{remark}

Given the above reformulation techniques and modeling choices, we now
summarize the problem variants arising from different combinations of the
latency function and travel cost functions in Table~\ref{tab:problem-types}.
Note that the MINLP reformulations of the congestion model under the user
equilibrium~\eqref{eq:minlp-ref:ue} and the system
optimum~\eqref{eq:minlp-ref:so} have the same structural properties, as
discussed in Remark~\ref{rem:structural-properties}.

\begin{table}
  \centering
  \caption{Properties of the MINLP reformulations of the congestion model
    depending on the choice of~$\beta \in \set{1,2,4}$ in the BPR travel cost
    functions and the choice of the latency function~$\latency(\flow)$, i.e.,
    maximum utilization~\eqref{eq:max-ratio-fun}, total
    utilization~\eqref{eq:sum-ratio-fun},
    or BPR~\eqref{eq:bpr-function-latency}.}
  \begin{tabular}{clcc}
    \toprule
    $\beta$ & $\latency(\flow)$ & convex feasible set & problem type \\
    \midrule
    $1$ & \eqref{eq:max-ratio-fun}, \eqref{eq:sum-ratio-fun} & \check & MILP  \\
    $1$ & \eqref{eq:bpr-function-latency} & \check & MIQP \\
    $2$ & \eqref{eq:max-ratio-fun}, \eqref{eq:sum-ratio-fun},
          \eqref{eq:bpr-function-latency} & \cross & MIQP \\
    $4$ & \eqref{eq:max-ratio-fun}, \eqref{eq:sum-ratio-fun},
          \eqref{eq:bpr-function-latency} & \cross & MIQP \\
    \bottomrule
  \end{tabular}
  \label{tab:problem-types}
\end{table}

Finally, let us briefly comment on the monotonicity of the latency function
with respect to the travel demand. We note that increasing the travel demand of
a single commodity does not necessarily lead to higher congestion levels, even
if the travel cost functions are linear. The following example illustrates that
higher travel demand may in fact reduce congestion levels due to route
reallocation effects. For related monotonicity phenomena and paradoxes in
congestion games, we also refer to \textcite{Cominetti_et_al:2024}.

\begin{example}
  \label{ex:paradox}
  We consider the network shown in Figure~\ref{fig:paradox}, which is taken
  from Example~3 in \textcite{Cominetti_et_al:2024}.
  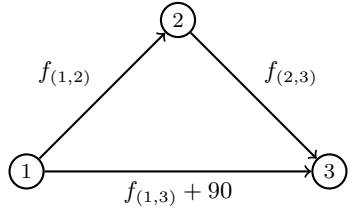
\begin{figure}
  \centering
  \begin{tikzpicture}[
    transform shape, thick,
    every node/.style={font=\small},
    vertex/.style={circle, draw=black, text=black, minimum
      size=0.45cm, font=\footnotesize\sffamily, inner sep=1pt},
    edge/.style={draw=black, line width=0.8pt}
    ]
    
    \node[vertex] (1) at (0, 0) {$1$};
    \node[vertex] (2) at (2, 2) {$2$};
    \node[vertex] (3) at (4, 0) {$3$};
    
    \draw[->,thick] (1) -- (2) node[pos=0.5,above left] {$\flow_{(1,2)}$};
    \draw[->,thick] (1) -- (3) node[pos=0.5,below] {$\flow_{(1,3)} + 90$};
    \draw[->,thick] (2) -- (3) node[pos=0.5,above right] {$\flow_{(2,3)}$};
  \end{tikzpicture}
  \caption{The network considered in Example~\ref{ex:paradox}.
    Arc labels represent the travel cost functions associated with each arc.}%
  \label{fig:paradox}%
\end{figure}%
  We consider three OD pairs given by~$(1,2)$,
  $(2,3)$, and~$(1,3)$ with demand~$1$, $100$, and~$20$, respectively.
  The resulting equilibrium flows under the UE are
  $\flow_{(1,2)} = 4$, $\flow_{(1,3)} = 17$, and $\flow_{(2,3)} = 103$.
  Let us now assume that all arcs have a practical capacity of~$1$, i.e.,
  $\capacity_a = 1$ for all $a \in \arcs = \set{(1,2),(2,3),(1,3)}$.
  Considering the maximum utilization function~\eqref{eq:max-ratio-fun} as the
  congestion measure, the worst-case congestion is attained on arc~$(2,3)$,
  namely~$\flow_{(2,3)}/\capacity_{(2,3)} = 103$.
  As observed in \textcite{Cominetti_et_al:2024}, if we now increase the travel
  demand of commodity~$(1,2)$ from~$1$ to~$4$, some of its flow is
  reallocated. The new equilibrium flows are given by
  $\flow_{(1,2)} = 6$, $\flow_{(1,3)} = 18$, and $\flow_{(2,3)} = 102$, which
  leads to the worst-case congestion level of~$102 < 103$.
  Hence, increasing the travel demand of a commodity can in fact lead
  to reduced congestion.
\end{example}

The previous observation implies that the demand realization causing the
worst-case congestion in our framework does not necessarily lie on the boundary
of the uncertainty set. Due to route reallocation effects,
intermediate demand realizations can produce higher congestion, which contrasts
with classic robust optimization settings in which worst cases typically occur
at the boundary of the uncertainty~set.

\subsection{Uncertainty Set}
\label{sec:uncertainty-set}

We assume that the travel demand~$d = (d_k)_{k \in \commodities}$ is
uncertain. In this paper, we consider three approaches to model this
uncertainty: budgeted uncertainty, ellipsoidal uncertainty, and the hose model.

\subsubsection{Budgeted Uncertainty}
\label{sec:budget-uncertainty}

As it is unlikely that all demands simultaneously realize in a worst-case
manner, a common approach is to adopt a budgeted (or~$\Gamma$-robust)
uncertainty modeling
\parencite{Bertsimas_Sim:2003,Bertsimas_Sim:2004,Sim:2004}. This model has
been widely used in transportation contexts, including
network design \parencite{Mattia:2019,Mattia_Poss:2018}, traffic assignment
\parencite{Ordonez_Stier-Moses:2007,Ordonez_Stier-Moses:2010,Ito:2011}, and
network pricing \parencite{Beck_et_al:2024}.

In the budgeted uncertainty model, each commodity~$k$ has a nominal
demand \mbox{$\bar{d}_k \in \R_{\geq 0}$}, which corresponds to the demand
value provided in the original instance data, and a maximum
deviation~$\delta_k \in \R_{\geq 0}$ from the nominal value so
that~$d_k \in [\bar{d}_k - \delta_k, \bar{d}_k + \delta_k]$ holds for
all~$k \in \commodities$.
Each commodity may deviate by an arbitrary fraction of
its maximum deviation~$\delta_k$ from the nominal value as long as the
 budget~$\Gamma \in \set{0,\ldots,\Abs{\commodities}}$ for the aggregated
 percentage deviation is not exceeded.
Overall, this can be modeled using the uncertainty set
\begin{equation}
\label{eq:budget-uncertainty}
\uncertaintyset_{\text{budget}} = \Defset{d \in \R^{\Abs{\commodities}}}{%
  d_k = \bar d_k + \delta_k z_k,\, z_k \in [-1,1],\, k \in \commodities,\,
  \sum_{k \in \commodities} \Abs{z_k} \leq \Gamma}.
\end{equation}
Here, the parameter~$\Gamma$ controls the size of the uncertainty set.
For~$\Gamma = 0$, no uncertainty is taken into account, i.e., we consider the
deterministic case, whereas for~$\Gamma = \Abs{\commodities}$, all commodities
may simultaneously experience their maximum deviation.
Note that~$\uncertaintyset_{\text{budget}}$ is non-empty, convex and
compact.
By assuming~$\delta_k \leq \bar d_k$ for all $k \in \commodities$,
we ensure~$d \in \R^{\Abs{\commodities}}_{\geq 0}$, i.e.,
Assumption~\ref{as:uncertainty-set} is satisfied.

\subsubsection{Ellipsoidal Uncertainty}
\label{sec:ellipsoid-uncertainty}

Using an ellipsoidal uncertainty modeling, we replace the cardinality
constraint in the budgeted uncertainty set with a quadratic constraint on the
vector of deviations, also called perturbation vector in the following. The
ellipsoidal uncertainty set is thus given by
\begin{equation}
\label{eq:ellipsoid-uncertainty}
\uncertaintyset_{\text{ellipsoid}} = \Defset{d \in \R^{\Abs{\commodities}}}{%
  d_k = \bar d_k + \delta_k z_k,\, k \in \commodities,\, \| z \|_2 \leq \rho},
\end{equation}
where~$\rho > 0$ is the radius of the ellipsoid
and $z = (z_k)_{k \in \commodities}$  is the perturbation vector.
To have a fair comparison between the ellipsoidal and the budgeted uncertainty
models, we select~$\rho$ such that the uncertainty
sets~$\uncertaintyset_{\text{budget}}$
and~$\uncertaintyset_{\text{ellipsoid}}$ cover a comparable range of demand
realizations. As shown in \textcite{Li_et_al:2011}, 
there exists a geometric relationship between these two sets, which depends on
the choice of the parameters~$\rho$ and~$\Gamma$.
In their analysis, they describe two cases:
\begin{enumerate}[label=(\roman*)]
\item If~$\rho = \Gamma$, the ellipsoid is the smallest one containing the
  polyhedron of the budgeted uncertainty model.
\item If~$\rho = \Gamma / \sqrt{\Abs{\commodities}}$, the ellipsoid is the
  largest one contained in the polyhedron of the budgeted uncertainty model.
\end{enumerate}

In this paper, we introduce a third intermediate case with~$\rho =
\sqrt{\Gamma}$, which yields uncertainty sets of comparable size.
In Figure~\ref{fig:uncertainty-sets-comparison}, we illustrate all three
aforementioned cases.
\begin{figure}
  \centering
  \includegraphics[width=\textwidth]{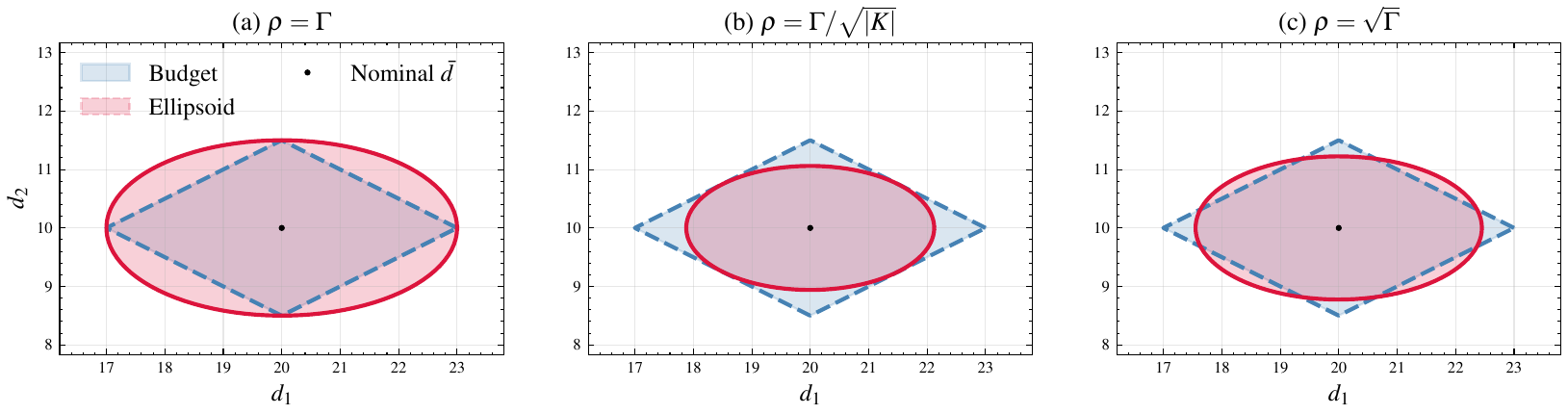}
  \caption{Comparison of budgeted and ellipsoidal uncertainty sets.}
  \label{fig:uncertainty-sets-comparison}
\end{figure}
As before, we assume that~$\delta_k \leq \bar d_k$ holds for
all~$k \in \commodities$ so that the set~$\uncertaintyset_{\text{ellipsoid}}$
satisfies Assumption~\ref{as:uncertainty-set}.

\subsubsection{Hose Model}
\label{sec:hose-uncertainty}

The hose model, originally introduced in the telecommunications network design
literature \parencite{hose,hose2}, aggregates demand uncertainty at the node
level rather than at the commodity level. This model has been applied in
various network design contexts; see, e.g.,
\textcite{Altin:2007,oriolo,mattiarnl}.
For each node~$v \in \nodes$, a bound~$b_v$ for the total traffic generated by
that node is specified. In this paper, we consider the so-called symmetric hose
uncertainty set given by
\begin{equation}
\label{eq:hose-uncertainty}
\uncertaintyset_{\text{hose}} = \Defset{d \in \R^{\Abs{\commodities}}_{\geq 0}}{%
  \sum_{\Defset{k \in \commodities}{s_k = v \text{ or } t_k = v}}
  d_k \leq b_v,\, v \in \nodes},
\end{equation}
where~$s_k$ and~$t_k$ denote the source and the destination node of
commodity~$k$, respectively.
%
To ensure a fair comparison between the hose model and the budgeted and
ellipsoidal uncertainty sets, we determine the node bounds as follows
\begin{equation*}
b_v= \sum_{k \in \commodities_v} \bar{d}_k + \sum_{j=1}^{\min \set{\Gamma,\,
    \Abs{\commodities_v}}} \delta_k^{(j)},
\end{equation*}
where~$\commodities_v$ is the set of commodities with source or destination in
node~$v \in \nodes$ and~$\delta_k^{(j)}$ are the deviations of the commodities
in~$\commodities_v$ sorted in non-increasing order.
Hence, we consider the~$\Gamma$ largest deviations among the commodities
in~$\commodities_v$.
By construction, we have $\uncertaintyset_{\text{budget}} \subseteq
\uncertaintyset_{\text{hose}}$ and~$\uncertaintyset_{\text{hose}}$ satisfies
Assumption~\ref{as:uncertainty-set}.


\section{Computational Setup and Experimental Design}
\label{sec:computational-results}

In this section, we discuss the setup and design of our computational
study. In Section~\ref{sec:implementation}, we elaborate on the
hardware, the software, and the solver used in our computational study.
Afterward, in Section~\ref{sec:instances}, we describe our test instances.

\subsection{Computational Setup}
\label{sec:implementation}

All experiments were performed on a 64-bit Linux system with an Intel Core
i9-13900K (\SI{3}{\giga\hertz}) and \SI{128}{\giga\byte} RAM using 4 threads.
The congestion models were implemented in \textsf{Python}~3.10.15 and solved
with \textsf{Gurobi} 12.0.1.
The time limit for each instance was set to \SI{2}{\hour}. Moreover,
we set the parameter \textsf{MIPGap} to $10^{-3}$ and disabled
presolve, as preliminary experiments indicated numerical
instabilities when using \textsf{Gurobi}'s default settings.
All other parameters were left at their default settings.
For each instance, we first solve the deterministic version of the problem and
use the solution to this problem as a warm start for solving the model under
demand uncertainty, which considerably improved runtime in preliminary
computational tests.

\subsection{Test Instances}
\label{sec:instances}

We consider instances of the Sioux Falls network
\parencite{LeBlanc_et_al:1975}, whose data is publicly available at
\url{https://github.com/bstabler/TransportationNetworks}, and instances from
the \textsf{SNDlib} \parencite{Orlowski_et_al:2010,SNDlib10}, which can be
accessed at \url{https://sndlib.put.poznan.pl}.

\subsubsection{Sioux Falls Instances}

We consider subnetworks of the Sioux Falls network obtained as
subgraphs induced by \SI{50}{\percent} or \SI{75}{\percent} of the
nodes with a varying number of arcs.
Specifically, we analyze five small instances with 12~nodes and five
medium instances with 18~nodes, which are illustrated in
Figure~\ref{fig:siouxfall_networks}.
\begin{figure}
  \centering
  \includegraphics[width=0.9\textwidth]{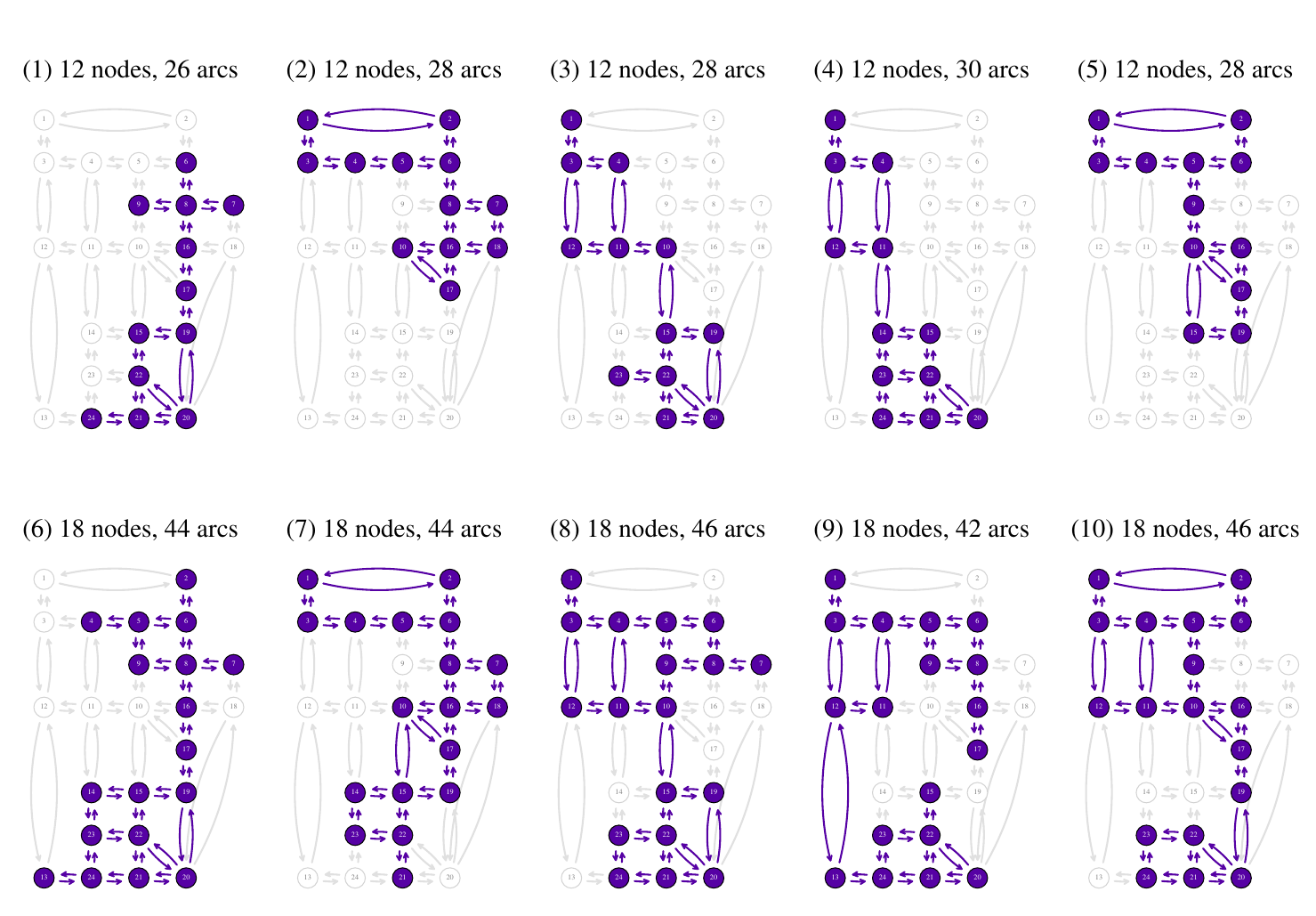}
  \caption{The subnetworks of the Sioux Falls network considered in the
    computational study. Nodes and arcs included in the subnetworks are shown
    in purple.}
  \label{fig:siouxfall_networks}
\end{figure}
For each subnetwork, we vary the number of
commodities~$\numcommodities \in \set{20, 30, 50, 75, 100}$ and randomly select
the desired number of commodities from the set of all origin-destination pairs.
Because the original Sioux Falls instance does not account for demand
uncertainty, we do the following.
For the $\Gamma$-robust formulation, the nominal demand~$\bar{d}_k$
for each commodity~$k \in \commodities$ is set to the demand value~$d_k$
reported in the original instance data. We then
set~$\delta_k = 0.25\bar{d}_k$ for the maximum deviation.
The parameters of the ellipsoidal and hose uncertainty sets are derived
from~$\bar{d}_k$ and~$\delta_k$ as described in
Section~\ref{sec:uncertainty-set}.
All values are scaled by a common factor to avoid numerical instabilities.
Finally, we note that all generated subnetworks are strongly connected, whereas
their underlying undirected graphs are not biconnected.
This structure allows us to assess the effectiveness of the enhancement
techniques introduced in Section~\ref{sec:tightening}.

\subsubsection{\textsf{SNDlib} Instances}

As the \textsf{SNDlib} instances have originally been created for
telecommunication network design, they do not include all parameters required
for traffic assignment problems, such as arc capacities and free-flow travel
times. Therefore, we proceed as follows.
From each \textsf{SNDlib} instance, we take the set of nodes~$\nodes$
with their geographical coordinates, the set of links, and the traffic demands.
We then parse the graph topology from the \textsf{SNDlib} native format, 
duplicating each undirected link to create a directed bidirectional graph. 
Arc capacities are set proportionally to the total
demand~$D = \sum_{k \in \commodities} d_k$ across all commodities and we
introduce heterogeneity through a random perturbation. More specifically, we
consider~$\capacity_a = 0.1D\xi_a$ for all~$a \in \arcs$,
where~$\xi_a \sim U[0.1, 2.0]$ is a uniformly distributed random value that
introduces variability across arcs.
Finally, to set the free-flow travel time~$c^{\fix}_a$ for each arc~$a = (i,j)$, 
we consider the Euclidean distance~$d_{ij}$ between its endpoints, which is
computed using the geographical coordinates provided in the \textsf{SNDlib}
data, and normalize it using the maximum distance~$d_{\max}$ in the network.
The free-flow time is then obtained as
\begin{equation*}
  c^{\fix}_a = c_{\min} + \frac{d_{ij}}{d_{\max}} \left( c_{\max} - c_{\min}
  \right),
  \quad a \in \arcs,
\end{equation*}
with~$c_{\min} = 1$ and~$c_{\max} = 50$.
Hence, longer arcs have higher free-flow costs, whereas shorter arcs have
lower costs.
As with the Sioux Falls instances, we vary the number of commodities
for each network. We consider five cases by selecting \SI{10}{\percent},
\SI{25}{\percent}, \SI{50}{\percent}, \SI{75}{\percent}, and \SI{100}{\percent}
of the number of commodities~$\numcommodities$ of the original instance.
Commodities are randomly selected from the set of all OD pairs and
the demand parameters are determined following the same procedure as for the
Sioux Falls instances.
We summarize the main characteristics of the considered \textsf{SNDlib}
instances in Table~\ref{tab:sndlib_instances} and illustrate their
network topologies in Figure~\ref{fig:sndlib_networks}.

\begin{table}
  \centering
  \caption{The number of nodes, arcs, and commodities of the considered
    \textsf{SNDlib} instances.}
  \label{tab:sndlib_instances}
  \begin{tabular}{lrrrrrrrr}
    \toprule
    Instance & $\Abs{\nodes}$ & $\Abs{\arcs}$
    & $\lceil\SI{10}{\percent}\numcommodities\rceil$
    & $\lceil\SI{25}{\percent}\numcommodities\rceil$
    & $\lceil\SI{50}{\percent}\numcommodities\rceil$
    & $\lceil\SI{75}{\percent}\numcommodities\rceil$
    & $\numcommodities$ \\
    \midrule
    \textsf{abilene}  & 12 & 30 & 13 & 33 & 66  & 99  & 132 \\
    \textsf{atlanta}  & 15 & 44 & 21 & 52 & 105 & 158 & 210 \\
    \textsf{nobel-us} & 28 & 42 & 9  & 22 & 45  & 68  & 91 \\
    \textsf{pdh}      & 11 & 68 & 2  & 6  & 12  & 18 & 24 \\
    \bottomrule
  \end{tabular}
\end{table}

\begin{figure}
  \centering
  \includegraphics[width=0.9\textwidth]{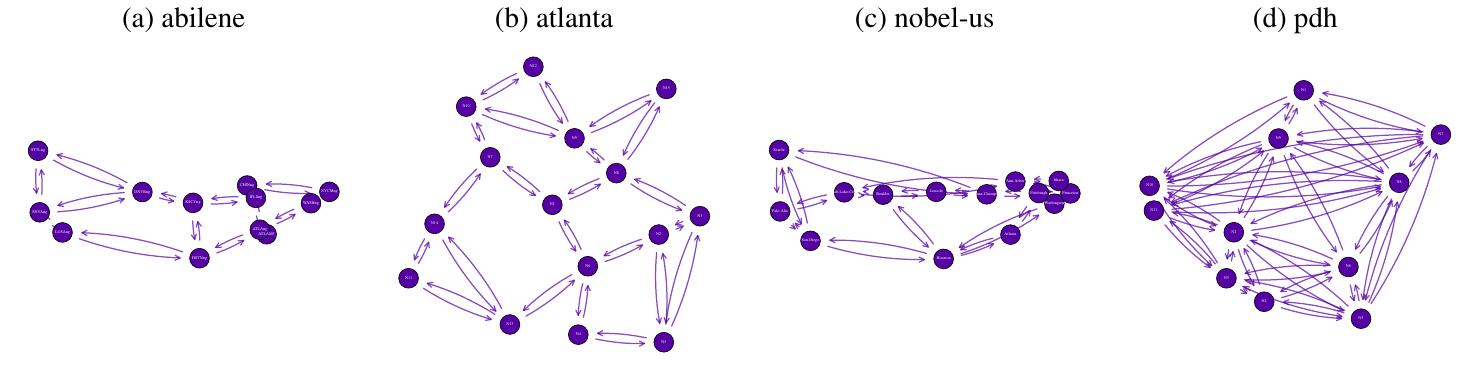}
  \caption{The network topologies of the considered \textsf{SNDlib} instances.}
  \label{fig:sndlib_networks}
\end{figure}

\section{Discussion of the Computational Results}
\label{sec:results}

We now present and discuss the results of our computational study.
Unless stated otherwise, the results refer to the Sioux Falls subnetworks under
budgeted uncertainty using the user equilibrium
formulation~\eqref{eq:minlp-ref:ue} and the enhancement techniques described in
Section~\ref{sec:tightening}.

In Section~\ref{sec:performance-analysis}, we discuss
the computational performance and scalability of our models.
In Section~\ref{sec:sensitivity}, we analyze the sensitivity of optimal
solutions to different choices of travel cost and latency functions.
In Section~\ref{sec:ue-so-comparison}, we compare the UE and
SO formulations in terms of congestion levels and computational
performance, whereas the impact of different uncertainty models is discussed in
Section~\ref{sec:uncertainty-analysis}.
Finally, in Section~\ref{subsubsec:sndlib-results}, we present the results for
the \textsf{SNDlib} instances.

For the ease of presentation, we abbreviate the
maximum utilization and total utilization latency functions as
\textsf{max\_ratio} and \textsf{sum\_ratio}, respectively.
Moreover, we refer to the travel cost functions with~$\beta = 1$,
$\beta = 2$, and $\beta = 4$ as linear, quadratic, and
BPR, respectively.

\subsection{Performance analysis}
\label{sec:performance-analysis}

We first evaluate the effectiveness of the enhancement techniques from
Section~\ref{sec:tightening}, along with the scalability and tractability of
the congestion models across different combinations of cost and congestion
functions.

\subsubsection{Impact of Enhancement Techniques}
\label{sec:enhancement-impact}

In what follows, we use \emph{tightened formulation} to denote the formulation
obtained after applying the enhancement techniques described in
Section~\ref{sec:tightening}, which reduce the model size and strengthen the
formulation by exploiting the network topology and structural properties of
optimal flows. We use the term \emph{standard formulation} to denote the
original formulation without any enhancements.

Table~\ref{tab:enhancement_comparison} reports the impact of the enhancement
techniques on the computation time and the number of solved instances 
for~$\numcommodities \in \set{20, 30, 50, 75, 100}$.
Here, we aggregate the results over the three considered latency functions and
report them by travel cost function, as the travel cost function
has the largest impact on computation time.
The latter is discussed in more detail in Section~\ref{subsubsec:scalability}.
To compare the tightened and standard formulations, we further report the
speed-up ratio in Table~\ref{tab:enhancement_comparison}, defined as the
execution time of the tightened formulation divided by that of the standard
formulation. For a fair comparison, we only present runtime results for
instances solved to optimality under both formulations.

\begin{table}
  \centering
  \caption{The number of instances solved to optimality
    (``\#opt'') and the average execution time (``time'', in
    \si{\second}) for the tightened and standard formulations of the congestion
    model for different values of $\numcommodities$ and travel cost 
    functions. Additionally, the speed-up ratio
    (``speed-up''), the time saved (in \si{\percent}),
    and the difference in the number of solved instances (``$\Delta$\#opt'') are
    shown.}
  \label{tab:enhancement_comparison}
  \ifPreprint
    \resizebox{\textwidth}{!}{%
    \begin{tabular}{clrrrrrrr}
      \toprule
      & & \multicolumn{2}{c}{\textbf{Tightened}}
      & \multicolumn{2}{c}{\textbf{Standard}}
      & \multicolumn{3}{c}{\textbf{Comparison}} \\
      \cmidrule(lr){3-4} \cmidrule(lr){5-6} \cmidrule(lr){7-9}
      $\numcommodities$ & $c(\flow)$ & \#opt & time & \#opt & time & speed-up & time saved & $\Delta$\#opt \\
      \midrule
      $20$ & & \textbf{90} & \textbf{0.08} & \textbf{90} & \textbf{1.77} & \textbf{21.34} & \textbf{95} & \textbf{0} \\
      & {linear} & 30 & 0.03 & 30 & 1.12 & 32.56 & 97 & 0 \\
      & {quadratic} & 30 & 0.06 & 30 & 1.95 & 34.37 & 97 & 0 \\
      & {BPR} & 30 & 0.16 & 30 & 2.25 & 14.22 & 93 & 0 \\
      \midrule
      $30$ & & \textbf{90} & \textbf{0.29} & \textbf{90} & \textbf{9.97} & \textbf{34.53} & \textbf{97} & \textbf{0} \\
      & {linear} & 30 & 0.07 & 30 & 2.95 & 40.55 & 98 & 0 \\
      & {quadratic} & 30 & 0.17 & 30 & 6.68 & 38.85 & 97 & 0 \\
      & {BPR} & 30 & 0.62 & 30 & 20.30 & 32.63 & 97 & 0 \\
      \midrule
      $50$ & & \textbf{90} & \textbf{2.14} & \textbf{89} & \textbf{59.06} & \textbf{27.61} & \textbf{96} & \textbf{1} \\
      & {linear} & 30 & 0.23 & 30 & 4.68 & 20.56 & 95 & 0 \\
      & {quadratic} & 30 & 1.04 & 30 & 17.89 & 17.16 & 94 & 0 \\
      & {BPR} & 30 & 5.25 & 29 & 157.90 & 30.08 & 97 & 1 \\
      \midrule
      $75$ & & \textbf{90} & \textbf{16.56} & \textbf{77} & \textbf{125.97} & \textbf{7.61} & \textbf{87} & \textbf{13} \\
      & {linear} & 30 & 0.96 & 30 & 12.61 & 13.16 & 92 & 0 \\
      & {quadratic} & 30 & 5.14 & 28 & 69.60 & 13.55 & 93 & 2 \\
      & {BPR} & 30 & 58.02 & 19 & 388.03 & 6.69 & 85 & 11 \\
      \midrule
      $100$ & & \textbf{81} & \textbf{27.44} & \textbf{64} & \textbf{438.14} & \textbf{15.96} & \textbf{94} & \textbf{17} \\
      & {linear} & 30 & 4.86 & 30 & 407.82 & 83.84 & 99 & 0 \\
      & {quadratic} & 30 & 26.43 & 26 & 345.65 & 13.08 & 92 & 4 \\
      & {BPR} & 21 & 127.98 & 8 & 911.62 & 7.12 & 86 & 13 \\
      \midrule
      \textbf{Total} & & \textbf{441} & \textbf{7.89} & \textbf{410} & \textbf{106.64} & \textbf{13.51} & \textbf{93} & \textbf{31} \\
      \bottomrule
    \end{tabular}}
\else
  \begin{tabular}{clrrrrrrr}
      \toprule
      & & \multicolumn{2}{c}{\textbf{Tightened}}
      & \multicolumn{2}{c}{\textbf{Standard}}
      & \multicolumn{3}{c}{\textbf{Comparison}} \\
      \cmidrule(lr){3-4} \cmidrule(lr){5-6} \cmidrule(lr){7-9}
      $\numcommodities$ & $c(\flow)$ & \#opt & time & \#opt & time & speed-up & time saved & $\Delta$\#opt \\
      \midrule
      $20$ & & \textbf{90} & \textbf{0.08} & \textbf{90} & \textbf{1.77} & \textbf{21.34} & \textbf{95} & \textbf{0} \\
      & {linear} & 30 & 0.03 & 30 & 1.12 & 32.56 & 97 & 0 \\
      & {quadratic} & 30 & 0.06 & 30 & 1.95 & 34.37 & 97 & 0 \\
      & {BPR} & 30 & 0.16 & 30 & 2.25 & 14.22 & 93 & 0 \\
      \midrule
      $30$ & & \textbf{90} & \textbf{0.29} & \textbf{90} & \textbf{9.97} & \textbf{34.53} & \textbf{97} & \textbf{0} \\
      & {linear} & 30 & 0.07 & 30 & 2.95 & 40.55 & 98 & 0 \\
      & {quadratic} & 30 & 0.17 & 30 & 6.68 & 38.85 & 97 & 0 \\
      & {BPR} & 30 & 0.62 & 30 & 20.30 & 32.63 & 97 & 0 \\
      \midrule
      $50$ & & \textbf{90} & \textbf{2.14} & \textbf{89} & \textbf{59.06} & \textbf{27.61} & \textbf{96} & \textbf{1} \\
      & {linear} & 30 & 0.23 & 30 & 4.68 & 20.56 & 95 & 0 \\
      & {quadratic} & 30 & 1.04 & 30 & 17.89 & 17.16 & 94 & 0 \\
      & {BPR} & 30 & 5.25 & 29 & 157.90 & 30.08 & 97 & 1 \\
      \midrule
      $75$ & & \textbf{90} & \textbf{16.56} & \textbf{77} & \textbf{125.97} & \textbf{7.61} & \textbf{87} & \textbf{13} \\
      & {linear} & 30 & 0.96 & 30 & 12.61 & 13.16 & 92 & 0 \\
      & {quadratic} & 30 & 5.14 & 28 & 69.60 & 13.55 & 93 & 2 \\
      & {BPR} & 30 & 58.02 & 19 & 388.03 & 6.69 & 85 & 11 \\
      \midrule
      $100$ & & \textbf{81} & \textbf{27.44} & \textbf{64} & \textbf{438.14} & \textbf{15.96} & \textbf{94} & \textbf{17} \\
      & {linear} & 30 & 4.86 & 30 & 407.82 & 83.84 & 99 & 0 \\
      & {quadratic} & 30 & 26.43 & 26 & 345.65 & 13.08 & 92 & 4 \\
      & {BPR} & 21 & 127.98 & 8 & 911.62 & 7.12 & 86 & 13 \\
      \midrule
      \textbf{Total} & & \textbf{441} & \textbf{7.89} & \textbf{410} & \textbf{106.64} & \textbf{13.51} & \textbf{93} & \textbf{31} \\
      \bottomrule
  \end{tabular}
  \fi
\end{table}

The results show that, for~$\numcommodities \in \set{20,30}$, all instances are
solved to optimality within the time limit of~\SI{2}{\hour} under both
formulations, with the tightened formulation reducing execution time by more
than \SI{93}{\percent}.
As~$\numcommodities$ increases, the enhancement techniques also affect
solvability. When incorporating our enhancement techniques, the number of
additional instances solved to optimality increases from~$1$
at~$\numcommodities=50$ to~$13$ at~$\numcommodities=75$ to~$17$
at~$\numcommodities=100$.
In summary, the tightened formulation leads to faster execution times and more
instances solved to optimality, especially when considering many
commodities and the BPR travel cost function.

\subsubsection{Scalability Analysis}
\label{subsubsec:scalability}

We now assess the scalability of the congestion models across different
combinations of travel cost functions and congestion measures.
Figure~\ref{fig:scalability} illustrates the results by comparing
execution times with the size of the congestion model (Figure~(a)) and with the
number of commodities (Figure~(b)).
Model size refers to the number of variables multiplied by the number
of constraints.

\begin{figure}
  \centering
  \includegraphics[width=0.85\textwidth]{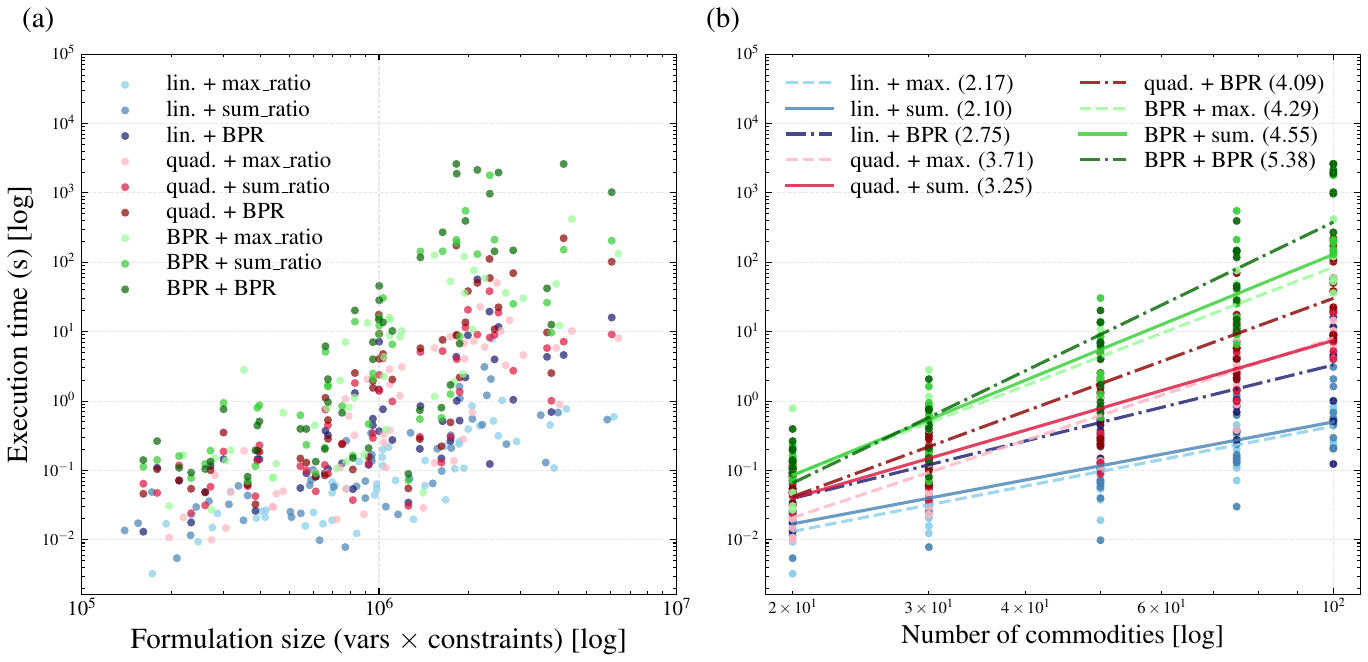}
  \caption{Scalability analysis for different combinations of travel cost and
    latency functions.
    Figure~(a): Execution time versus problem size.
    Figure~(b): Execution time versus number of commodities~$\numcommodities$
    with fitted power-law models~$t = a \cdot \numcommodities^\alpha$
    and~$\alpha$ being the growth exponent reported in brackets in the
    legend.}
\label{fig:scalability}
\end{figure}

Both plots in Figure~\ref{fig:scalability} show a considerable dispersion in
execution times and, for comparable model sizes and identical values
of~$\numcommodities$, execution times vary by several orders of magnitude. This
indicates that the problem size alone does not fully explain the computational
behavior and that structural properties of the formulation play a key role in
scalability. In this context, we observe that the dominant factor
affecting scalability is the choice of the travel cost function.
Instances with {BPR} costs exhibit the highest execution times and the
steepest growth, whereas {linear} costs lead to the lowest
execution times across all combinations.
{Quadratic} costs reveal an intermediate behavior. 
These trends are visible in Figure~\ref{fig:scalability}~(a)
through the vertical separation of color groups and in
Figure~\ref{fig:scalability}~(b) via larger scaling exponents for
{BPR} costs.

In contrast, the impact of the latency function is less
pronounced. For a given cost function, differences among \textsf{sum\_ratio},
\textsf{max\_ratio}, and \textsf{BPR} congestion models
result in smaller variations in execution time, as illustrated by the
overlap of points of the same color but different shades in
Figure~\ref{fig:scalability}~(a) and the proximity 
of curves with the same color but different line styles in
Figure~\ref{fig:scalability}~(b).
Nevertheless, for all three cost functions, we observe that the BPR congestion
measure yields the highest execution times, followed by \textsf{max\_ratio},
whereas \textsf{sum\_ratio} provides the fastest solution times.
Overall, these results indicates that the computational scalability is
primarily determined by the travel cost function, whereas the latency function
plays a secondary, yet still noticeable, role in execution times.

\subsubsection{Performance Across Function Combinations}
\label{subsubsec:performance-combinations}

As observed in Section~\ref{subsubsec:scalability}, the choice of the travel
cost function~$c(\flow)$ has a stronger impact on the solution time of
the congestion models than the choice of latency function. To quantify
these differences, we now aggregate performance metrics across all instances
solved to optimality for each combination of travel cost and latency function.
In Table~\ref{tab:performance_metrics}, we summarize the average execution
times and success rates for each combination. Here and in what follows,
success rate refers to the percentage of instances solved to optimality
out of all considered instances.
The results confirm that the travel cost function is the primary factor
affecting computational performance.
Although latency functions also influence performance, there
is no consistent trend that holds across all configurations.
For example, with {linear} or {quadratic} costs, all
instances are solved to optimality and execution times for \textsf{max\_ratio}
are comparable to those of \textsf{sum\_ratio} and faster
than \textsf{BPR}. However, for {BPR} travel costs,
\textsf{max\_ratio} leads to more challenging optimization problems,
increasing the fraction of instances that cannot be solved to optimality within
the time limit.

\begin{table}
  \centering
  \caption{The average execution time (in \si{\second}) over all instances
    solved to optimality and the success rate (in \si{\percent}).}
  \label{tab:performance_metrics}
  \begin{tabular}{lrrrrrr}
    \toprule
    & \multicolumn{3}{c}{\bf Execution time} & \multicolumn{3}{c}{\bf Success rate} \\
    \cmidrule(lr){2-4} \cmidrule(lr){5-7}
    $c(\flow)$ & \textsf{max\_ratio} & \textsf{sum\_ratio} & \textsf{BPR}
                                     & \textsf{max\_ratio} & \textsf{sum\_ratio} & \textsf{BPR} \\
    \midrule
    {linear}    & 0.33  & 0.32   & 3.05   & 100.0 & 100.0 & 100.0 \\
    {quadratic} & 3.72  & 3.65   & 19.23  & 100.0 & 100.0 & 100.0 \\
    {BPR}       & 27.52 & 83.42  & 303.24 & 92.0  & 94.0  & 96.0 \\
    \bottomrule
  \end{tabular}
\end{table}

\subsection{Sensitivity of Network Congestion to Travel Cost and Latency
  Functions}
\label{sec:sensitivity}

We now examine the impact of travel cost and latency functions on congestion
estimates. To ensure a meaningful comparison, we focus on those instances that
were solved to optimality for all nine combinations of travel cost and latency 
functions for each value of~$\numcommodities$.

\subsubsection{Impact of the Travel Cost Function on Congestion Estimates}

In Figure~\ref{fig:objective_analysis}, we show the evolution of congestion
levels as a function of the number of commodities~$\numcommodities$ for each
combination of travel cost and latency function.
\begin{figure}
  \centering
  \includegraphics[width=0.95\textwidth]{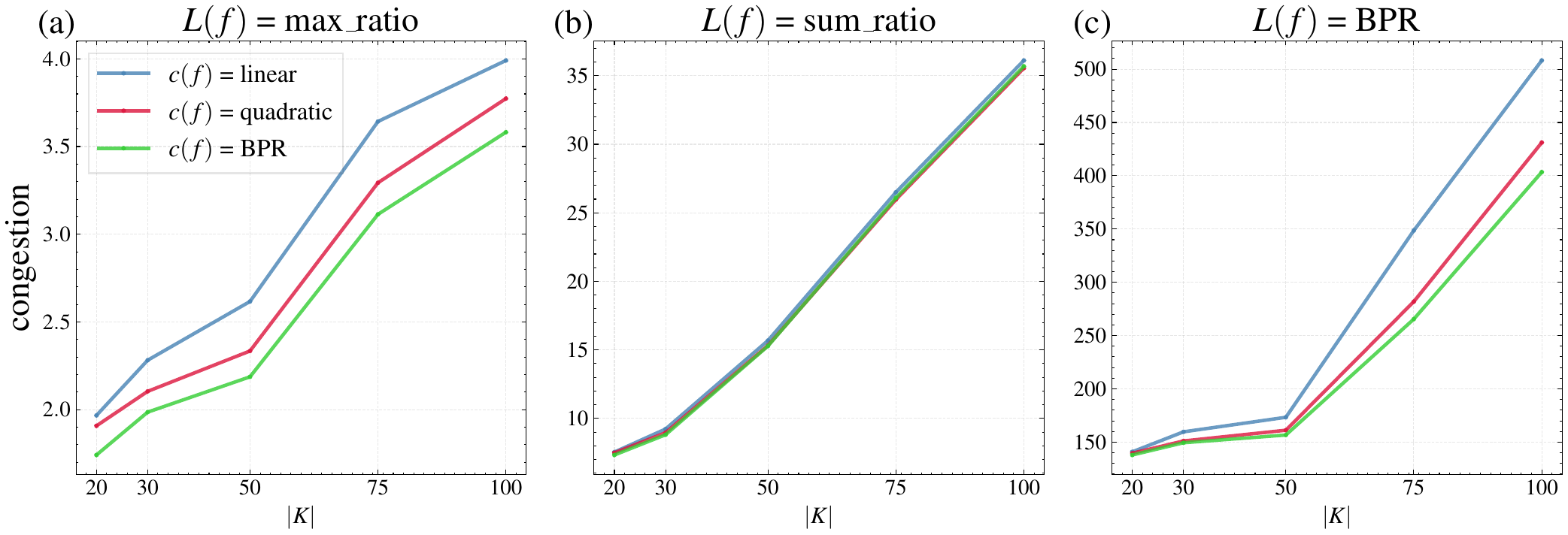}
  \caption{The congestion estimates for the models using the
    \textsf{max\_ratio} (a), \textsf{sum\_ratio} (b), and 
    \textsf{BPR} (c) congestion measures. Each
    panel shows how the average congestion level varies with the number of
    commodities for different cost functions.}
  \label{fig:objective_analysis}
\end{figure}
For a given latency function~$\latency(\flow)$, the choice of cost
function~$c(\flow)$ noticeably affects the estimated congestion levels. In
particular, the {BPR} cost function consistently yields lower congestion
estimates than {linear} or {quadratic} costs.
Figure~\ref{fig:objective_analysis}~(a) illustrates this effect for the
\textsf{max\_ratio} latency function. Here,
{BPR} costs produce congestion estimates that are about
\SI{15}{\percent}--\SI{20}{\percent} lower than those obtained with
{linear} costs.
This indicates that nonlinear travel cost functions, which penalize congestion
more strongly, are more effective at limiting extreme arc saturation.
Figure~\ref{fig:objective_analysis}~(c) shows the same qualitative behavior
for the \textsf{BPR} latency function. Again, {BPR} costs achieve the
lowest congestion values and the effect is particularly pronounced for
instances with a large number of commodities~$\numcommodities$.

In contrast, Figure~\ref{fig:objective_analysis}~(b) shows that the
\textsf{sum\_ratio} metric is less sensitive to the choice of
travel cost function as all three of them yield similar
congestion estimates across the range of~$\numcommodities$.
This behavior is explained by the aggregate nature of the metric.
Different cost functions lead to different flow distributions at the arc
level. However, local variations tend to compensate when summed over all
arcs of the network as higher utilization of some arcs is balanced by lower
utilization of others.
The small size of the network further amplifies this effect as the limited
number of alternative paths constrains the degree to which
flows can be redistributed in response to changes in the cost function.

\subsubsection{Impact of the Latency Function on Congestion Estimates}

To assess how the choice of the latency function in the objective of
the congestion model affects congestion estimates, we now analyze how
solutions obtained by optimizing one latency function perform under another.
Specifically, for a given latency function~$\latency_{\mathrm{opt}}$, i.e.,
either \textsf{max\_ratio}, \textsf{sum\_ratio}, or \textsf{BPR}, we compute
an optimal flow~$\flow^*_{\latency_{\mathrm{opt}}}$ by solving the congestion
model~\eqref{eq:minlp-ref:ue} and then evaluate the resulting congestion
level using a different latency function~$\latency_{\mathrm{eval}}$.
For this purpose, we define the relative gap
\begin{equation}
  \label{eq:gap}
  \text{gap}(\latency_{\mathrm{opt}}, \latency_{\mathrm{eval}}) =
  \frac{\latency_{\mathrm{eval}}(\flow^*_{\latency_{\mathrm{eval}}}) -
    \latency_{\mathrm{eval}}(\flow^*_{\latency_{\mathrm{opt}}})}
  {\latency_{\mathrm{eval}}(\flow^*_{\latency_{\mathrm{eval}}})}
\end{equation}
for each pair~$(\latency_{\mathrm{opt}}, \latency_{\mathrm{eval}})$.
Here, $\flow^*_{\latency_{\mathrm{opt}}}$
and~$\flow^*_{\latency_{\mathrm{eval}}}$ denote the
flows that optimize the congestion measures~$\latency_{\mathrm{opt}}$
and~$\latency_{\mathrm{eval}}$, respectively.
By construction, we have \mbox{$\text{gap}(\latency_{\mathrm{opt}},
  \latency_{\mathrm{eval}}) \geq 0$}.
Gaps close to zero indicate that optimizing for~$\latency_{\mathrm{opt}}$
yields solutions that are near-optimal when evaluated
under~$\latency_{\mathrm{eval}}$, whereas larger values
indicate a greater discrepancy between the two formulations.
We note that this analysis does not account for the multiplicity of
equilibrium flows. Hence, other equilibria optimal for~$\latency_{\mathrm{opt}}$
may perform better under~$\latency_{\mathrm{eval}}$.
Here, we restrict ourselves to evaluating the solution returned by
\textsf{Gurobi}.

\begin{table}
  \centering
  \caption{The relative gap (in \si{\percent}) on~$\latency_{\mathrm{eval}}$
    when optimizing for $L_{\mathrm{opt}}$ by the number of
    commodities~$\numcommodities$. Values are averaged over all considered
    instances and cost functions.}
  \label{tab:gaps_full}
  \begin{tabular}{llrrrrr}
    \toprule
    & & \multicolumn{5}{c}{$\numcommodities$} \\
    \cmidrule(lr){3-7}
    $\latency_{\mathrm{eval}}$ & $\latency_{\mathrm{opt}}$ & $20$ & $30$ & $50$ & $75$ & $100$ \\
    \midrule
    \textsf{max\_ratio} & \textsf{sum\_ratio} & $3.1$ & $2.3$ & $3.8$ & $5.0$ & $7.8$ \\
                        & \textsf{BPR} & $1.6$ & $0.5$ & $0.7$ & $1.1$ & $2.9$
    \\
    \midrule
    \textsf{sum\_ratio} & \textsf{max\_ratio} & $2.3$ & $3.3$ & $3.2$ & $3.1$ & $6.0$ \\
                        & \textsf{BPR} & $1.3$ & $1.3$ & $1.2$ & $1.6$ & $1.4$
    \\
    \midrule
    \textsf{BPR} & \textsf{max\_ratio} & $0.1$ & $0.3$ & $0.9$ & $2.7$ & $10.2$ \\
                        & \textsf{sum\_ratio} & $0.2$ & $0.7$ & $2.1$ & $5.6$ & $7.8$ \\
    \bottomrule
  \end{tabular}
\end{table}

Table~\ref{tab:gaps_full} reports the gaps for all distinct pairs 
of~$(\latency_{\mathrm{eval}}, \latency_{\mathrm{opt}})$ and for each value
of~$\numcommodities$.
For instances with~$\numcommodities \in \set{20,30}$, the discrepancies
between latency functions are limited, with gap values remaining within
\SI{3.3}{\percent} for all pairs.
However, as~$\numcommodities$ increases, larger differences emerge.
For instance, optimizing with respect to \textsf{sum\_ratio} instead of
\textsf{max\_ratio} results in a gap of approximately \SI{8}{\percent}
for~$\numcommodities = 100$. Similarly,
optimizing for \textsf{max\_ratio} or \textsf{sum\_ratio} 
instead of \textsf{BPR} leads to gaps of up to \SI{10}{\percent}.
Conversely, flows obtained by optimizing for \textsf{BPR} exhibit smaller
discrepancies when evaluated under other latency functions, with gaps
below \SI{3}{\percent} in most cases.
These results indicate that \textsf{BPR}-optimized flows lead to maximum
utilization and total utilization levels that are close to their respective
worst case.
Overall, the magnitude of the discrepancy between latency functions increases
with the number of commodities.

\subsection{User equilibrium vs.\ system optimum}
\label{sec:ue-so-comparison}

We now compare the UE and SO formulations of the congestion model in terms of
both computational performance and obtained congestion levels.
As before, we focus on instances that were solved to optimality by both
formulations to ensure a fair comparison.

\subsubsection{Runtime Comparison}

In Figure~\ref{fig:ue_so_time}, we compare the computation times of the UE and
SO formulations. The scatter plot in Figure~\ref{fig:ue_so_time}~(a) shows
that solving the UE formulation consistently requires shorter execution times
than solving the SO formulation. This advantage  is particularly pronounced
for instances with a large number of commodities.
Figure~\ref{fig:ue_so_time}~(b) presents the average execution time ratios
for each combination of travel cost and latency function.
Ratios are computed as the execution time of UE divided by that of SO.
The mean ratio ranges from approximately $0.4$ to $0.6$, indicating that
solving the UE formulation requires only \SI{40}{\percent}--\SI{60}{\percent}
of the time needed for solving the SO formulation.
{Linear} cost functions exhibit the highest ratios (around~$0.6$),
whereas {quadratic} and {BPR} cost functions yield lower ratios.
This suggests that the computational advantage of solving the UE formulation
increases for more involved travel cost functions. In contrast, no consistent
pattern is observed with respect to the choice of the latency function.

\begin{figure}
  \centering
  \includegraphics[width=0.95\textwidth]{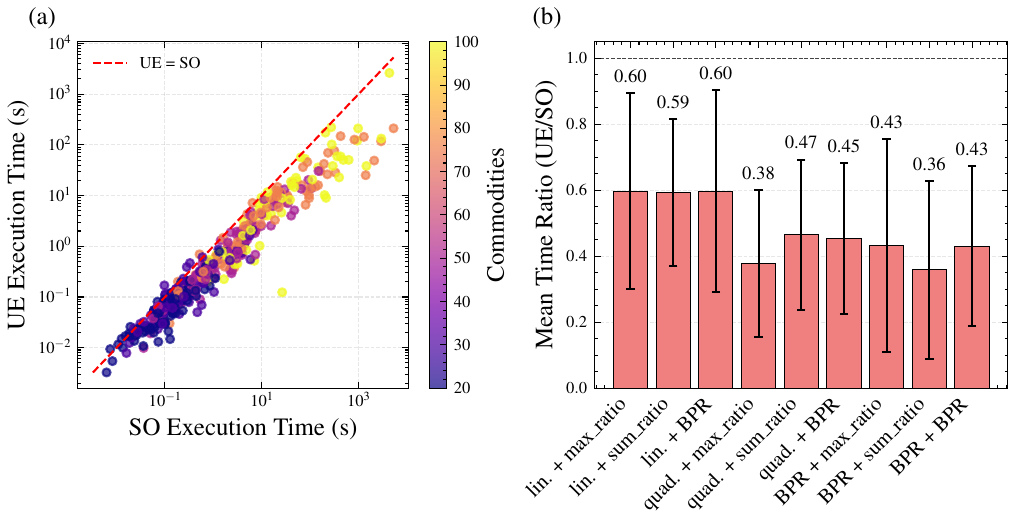}
  \caption{Computation time comparison between UE and SO formulations. 
    Figure~(a): Scatter plot of execution times (in \si{\second}, logarithmic
    scale). Each point represents a single instance, with colors indicating the
    number of commodities. The red dashed line marks instances for which
    solving the UE and SO formulation requires the same amount of time.
    Figure~(b): Mean execution time ratio (red box) with standard deviation
    (error bars) for each combination of travel cost and latency
    function. Mean values are shown above each error bar.}
  \label{fig:ue_so_time}
\end{figure}

\subsubsection{Comparison of Congestion Levels}

We now compare the congestion levels obtained by the UE and SO formulations by
analyzing the distribution of the congestion ratio (CR, see
Section~\ref{sec:congestion-ratio}) across all instances and function
combinations.
Figure~\ref{fig:congestion_ratio}~(a) shows a histogram of CR values. It can be
seen that most instances achieve CR values close to~$1$, indicating that UE
solutions typically yield congestion levels comparable to those of SO
solutions.
However, notable outliers exist, showing that UE solutions can lead to
significantly higher congestion than SO solutions in some cases.
At the same time, some instances have CR values below~$1$, which is consistent
with the behavior observed in Example~\ref{ex:cr-less-than-one}.
Figure~\ref{fig:congestion_ratio}~(b) illustrates the distribution of CR values
across different combinations of travel cost and latency functions.
As in Figure~\ref{fig:congestion_ratio}~(a), the distributions are generally
concentrated around~$1$. Combinations including the \textsf{max\_ratio}
congestion measure tend to exhibit higher congestion ratios, indicating that UE
solutions under this measure can produce more congestion relative to SO. In
contrast, the \textsf{sum\_ratio} measure yields the smallest CR values,
suggesting that UE solutions under this measure can sometimes lead to more
favorable outcomes.
Finally, combinations with {BPR} costs show a wider spread of CR values,
reflecting greater variability in congestion outcomes between the UE and SO
formulations.

\begin{figure}
  \centering
  \includegraphics[width=0.8\textwidth]{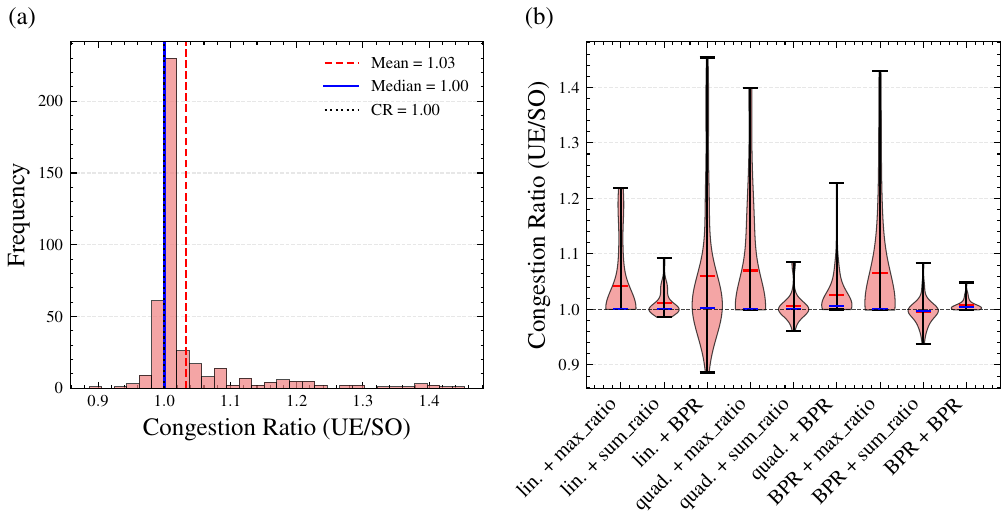}
  \caption{The distribution of the congestion ratio (CR) across all instances
    and function combinations. 
    Figure~(a): Histogram of CR values, with mean and median indicated by
    dashed red and solid blue lines, respectively.
    Figure~(b): Violin plots for the CR distribution for each 
    combination of travel cost and latency functions, with mean (red) and
    median (blue) markers.}
  \label{fig:congestion_ratio}
\end{figure}

\subsection{Uncertainty Set Comparison}
\label{sec:uncertainty-analysis}

We now compare the five uncertainty sets introduced in
Section~\ref{sec:uncertainty-set}:
the budgeted ($\Gamma$-robust) uncertainty set~$\uncertaintyset_{\text{budget}}$,
the ellipsoidal uncertainty set~$\uncertaintyset_{\text{ellipsoid}}$
with~$\rho \in \set{\sqrt{\Gamma},\, \Gamma,\, \Gamma/\sqrt{\numcommodities}}$,
and the symmetric hose uncertainty set~$\uncertaintyset_{\text{hose}}$.
For the ease of presentation, we restrict the analysis to instances derived
from the Sioux Falls network with $12$~nodes and $26$~arcs
(cf.~Figure~\ref{fig:siouxfall_networks}) for~$\numcommodities \in
\set{20,30,50,75,100}$ and all nine combinations of travel cost and latency
functions. This results in $45$~considered instances per uncertainty
set.


\subsubsection{Computational Tractability}
\label{sec:uncertainty-tractability}

Figure~\ref{fig:uncertainty_solved} illustrates the number of instances solved
to optimality within the time limit of~\SI{2}{\hour} for each uncertainty set
and combination of travel cost and latency functions.
\begin{figure}
  \centering
  \includegraphics[width=0.7\textwidth]{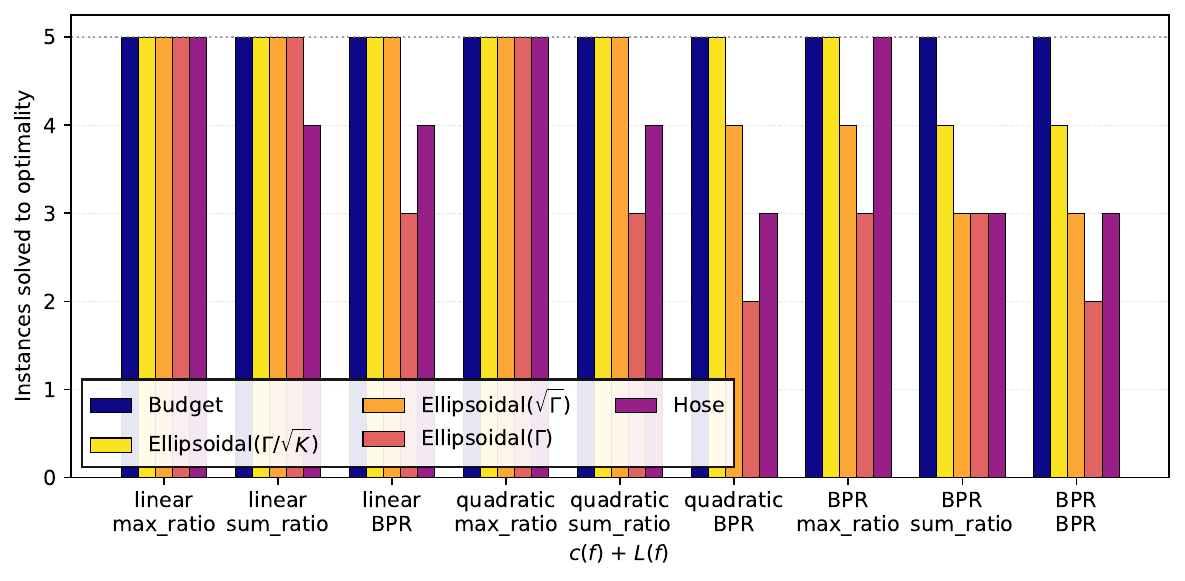}
  \caption{The number of instances solved to optimality by uncertainty set and
    function combination.}
  \label{fig:uncertainty_solved}
\end{figure}
The results indicate a clear dominance in terms of computational
tractability of the resulting congestion models: the budget uncertainty model
is the most computationally tractable, whereas the ellipsoidal model with~$\rho
= \Gamma$ seems to be the most challenging formulation.
In particular, the number of solved instances per formulation reflects the
computational challenges associated with both the geometric properties of the
uncertainty sets (e.g., their size) and the nature of the resulting
optimization problems (MILP vs.~MIQP).

Table~\ref{tab:uncertainty_time} further summarizes the average execution times
for each uncertainty model over the $30$ instances (out of $45$) that were
solved to optimality under all five uncertainty sets.
We observe that the budget uncertainty set consistently yields the shortest
solution times across all function combinations.
In contrast, the ellipsoidal model with~$\rho = \Gamma$ requires the longest
computation times, often exceeding the other formulations by two to three
orders of magnitude.

\begin{table}
  \centering
  \caption{The average computation times (in \si{\second}) by uncertainty set and
    function combination. Additionally, the number of instances solved to
    optimality (``\#opt'') is shown.}
  \label{tab:uncertainty_time}
  \begin{tabular}{llrrrrrr}
    \toprule
    & & & & \multicolumn{3}{c}{$\uncertaintyset_{\text{ellipsoid}}$} & \\
    \cmidrule(lr){5-7}
    $c(\flow)$
    & $\latency(\flow)$ & \#opt & $\uncertaintyset_{\text{budget}}$ & $\rho = \frac{\Gamma}{\sqrt{\numcommodities}}$ & $\rho = \sqrt{\Gamma}$ & $\rho = \Gamma$ & $\uncertaintyset_{\text{hose}}$ \\
    \midrule
    {linear}   && \textbf{12} & \bf{0.05} & \bf{0.60} & \bf{1.55} & \bf{150.72} & \bf{1.14} \\
            & \textsf{sum\_ratio}  &4  & 0.01 & 0.11 & 1.29 & 3.90 & 1.62 \\
            & \textsf{max\_ratio}  &5  & 0.07 & 0.83 & 1.97 & 24.34 & 0.87 \\
    & \textsf{BPR}         &3  & 0.06 & 0.87 & 1.18 & 557.09 & 0.94 \\
    \midrule
    {quadratic}&& \textbf{10} & \bf{1.20} & \bf{9.93} & \bf{33.30} & \bf{446.52} & \bf{5.54} \\
            & \textsf{sum\_ratio}  &3  & 0.05 & 1.75 & 1.99 & 4.36 & 1.16 \\
            & \textsf{max\_ratio}  &5  & 2.34 & 18.65 & 65.19 & 884.85 & 9.99 \\
    & \textsf{BPR}         &2  & 0.08 & 0.38 & 0.53 & 13.95 & 0.52 \\
    \midrule
    {BPR}      && \textbf{8} & \bf{0.05} & \bf{1.76} & \bf{1.68} & \bf{42.07} & \bf{21.34} \\
            & \textsf{sum\_ratio}  &3  & 0.65 & 3.17 & 2.85 & 98.01 & 55.90 \\
            & \textsf{max\_ratio}  &3  & 0.72 & 1.12 & 1.11 & 12.92 & 0.47 \\
            & \textsf{BPR}         &2  & 0.12 & 0.59 & 0.78 & 1.86 & 0.81 \\
    \bottomrule
  \end{tabular}
\end{table}


\subsubsection{Impact on Worst-Case Congestion}
\label{sec:uncertainty-congestion}

We now assess the impact of the uncertainty model on worst-case congestion by
comparing the congestion levels of the models with demand uncertainty to those
of the deterministic model.
In Figure~\ref{fig:uncertainty_heatmap}, we show heatmaps of the percentage
increase in congestion relative to the deterministic case for each
uncertainty model, disaggregated by travel cost and latency function.
We observe that all uncertainty models lead to higher worst-case
congestion compared to the deterministic setting.
The lowest worst-case congestion is observed for the ellipsoidal uncertainty
set with~$\rho = \Gamma/\sqrt{\numcommodities}$, followed by the budgeted
uncertainty set, the ellipsoidal sets with~$\rho = \sqrt{\Gamma}$ and~$\rho =
\Gamma$, and finally the hose model. This ordering aligns with the set
inclusions discussed in Section~\ref{sec:uncertainty-set}.
Distinct patterns also emerge with respect to the travel cost and latency
functions. Across all uncertainty sets, the \textsf{max\_ratio} congestion
measure produces the largest increases in worst-case congestion, whereas
\textsf{sum\_ratio} yields the smallest.
The hose model and the ellipsoidal uncertainty set with~$\rho = \Gamma$
consistently result in higher worst-case congestion.
We further observe that, for the budgeted model and ellipsoidal uncertainty
sets with~$\rho \in \set{\Gamma/\sqrt{\numcommodities},\, \sqrt{\Gamma}}$, the
choice of travel cost function has a limited effect on congestion levels.
In contrast, congestion levels are highly sensitive to the travel cost function
for the ellipsoidal set with~$\rho = \Gamma$ and the hose model. In particular,
the hose model produces increases exceeding \SI{50}{\percent} in most cases,
with peaks above \SI{90}{\percent} for \textsf{linear} and \textsf{quadratic}
cost functions.

\begin{figure}
  \centering
  \includegraphics[width=0.9\textwidth]{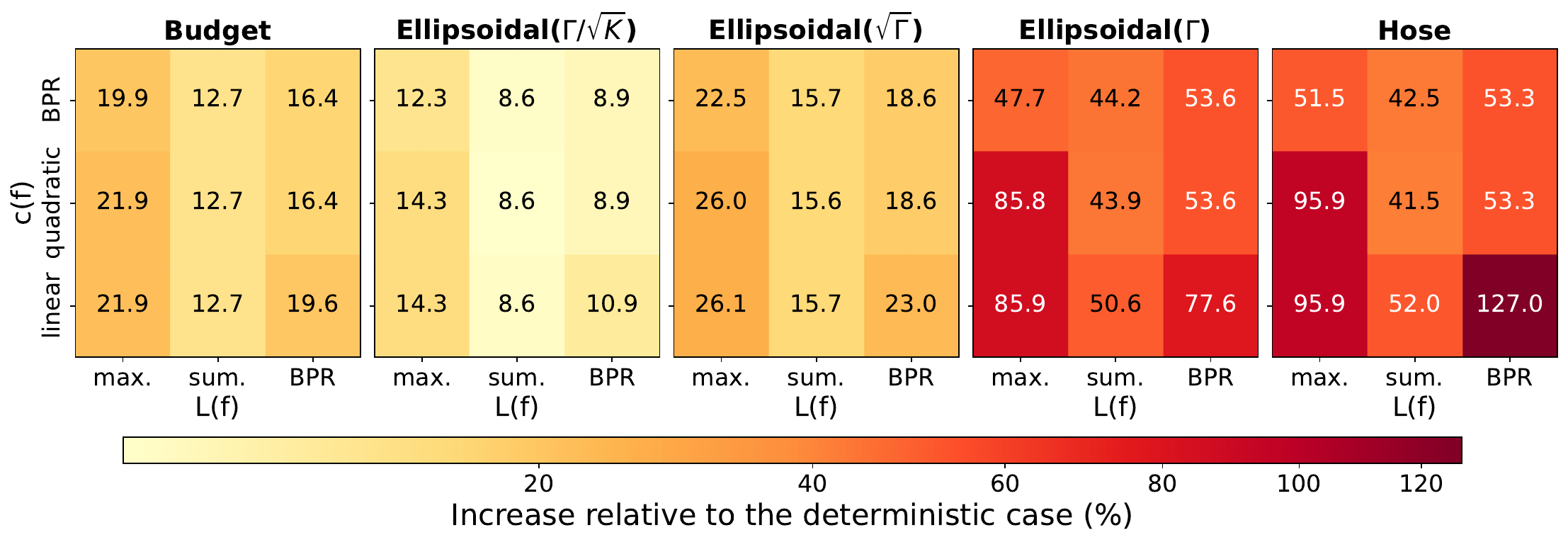}
  \caption{The percentage increase in worst-case congestion relative to the
    deterministic case, by uncertainty set and function combination.
    All panels share a common color scale for comparability. Darker colors
    indicate larger deviations from the deterministic solution.}
  \label{fig:uncertainty_heatmap}
\end{figure}

\subsection{Results for \textsf{SNDlib} Instances}
\label{subsubsec:sndlib-results}

We now discuss the computational results for the \textsf{SNDlib} instances.
These instances exhibit different network topologies and demand structures
compared to those of the Sioux Falls subnetworks, allowing us to evaluate the
congestion models from a broader perspective.

Table~\ref{tab:sndlib_by_graph} reports the number of solved instances and the
average execution times for each network across different numbers of
commodities.
For instances with \SI{10}{\percent} and \SI{25}{\percent} of the commodities
from the original network, almost all instances are solved to
optimality. However, as the number of commodities increases, solving the
congestion models becomes increasingly challenging.
Among the considered networks, \textsf{atlanta} is particularly challenging.
Only $5$ out of $9$ instances can be solved to optimality
for \SI{50}{\percent} of the commodities from the original network and we
observe that execution times are high even for smaller instances.
The same qualitative behavior can be observed for the \textsf{nobel-us} and
\textsf{pdh} networks: as the number of commodities increases, the number of
instances solved to optimality decreases.
Overall, \textsf{abilene} is the most computationally tractable network, with
all instances solved to optimality for up to \SI{50}{\percent} of the original
commodities, and $6$ out of $9$ instances solved for larger numbers of
commodities.

\begin{table}
  \centering
  \caption{The number of \textsf{SNDlib} instances solved to optimality
    (``\#opt'', out of $9$) and the average execution time (in \si{\second})
    for instances with \SI{10}{\percent}, \SI{25}{\percent}, \SI{50}{\percent},
    \SI{75}{\percent}, and \SI{100}{\percent} of the number of commodities from
    the original network.}
  \label{tab:sndlib_by_graph}
    \ifPreprint
    \resizebox{\textwidth}{!}{%
    \begin{tabular}{lrrrrrrrrrr}
      \toprule
       & \multicolumn{2}{c}{\SI{10}{\percent}} & \multicolumn{2}{c}{\SI{25}{\percent}} & \multicolumn{2}{c}{\SI{50}{\percent}} & \multicolumn{2}{c}{\SI{75}{\percent}} & \multicolumn{2}{c}{\SI{100}{\percent}} \\
      \cmidrule(lr){2-3} \cmidrule(lr){4-5} \cmidrule(lr){6-7}
      \cmidrule(lr){8-9} \cmidrule(lr){10-11}
       & \#opt & time & \#opt & time & \#opt & time & \#opt & time & \#opt & time \\
      \midrule
      \textsf{abilene}  & 9 & 2.2   & 9 & 7.0    & 9 & 11.9   & 6 & 24.9   & 6 & 28.0 \\
      \textsf{atlanta}  & 7 & 73.5  & 9 & 637.4  & 5 & 1993.4 & 3 & 29.9   & 4 & 1320.1 \\
      \textsf{nobel-us} & 9 & 1.5   & 9 & 67.1   & 4 & 1094.1 & 3 & 88.3   & 3 & 67.7 \\
      \textsf{pdh}      & 9 & 0.6   & 9 & 7.2    & 6 & 48.5   & 3 & 17.0   & 3 & 43.8 \\
      \bottomrule
    \end{tabular}}
\else
  \begin{tabular}{lrrrrrrrrrr}
      \toprule
       & \multicolumn{2}{c}{\SI{10}{\percent}} & \multicolumn{2}{c}{\SI{25}{\percent}} & \multicolumn{2}{c}{\SI{50}{\percent}} & \multicolumn{2}{c}{\SI{75}{\percent}} & \multicolumn{2}{c}{\SI{100}{\percent}} \\
      \cmidrule(lr){2-3} \cmidrule(lr){4-5} \cmidrule(lr){6-7}
      \cmidrule(lr){8-9} \cmidrule(lr){10-11}
       & \#opt & time & \#opt & time & \#opt & time & \#opt & time & \#opt & time \\
      \midrule
      \textsf{abilene}  & 9 & 2.2   & 9 & 7.0    & 9 & 11.9   & 6 & 24.9   & 6 & 28.0 \\
      \textsf{atlanta}  & 7 & 73.5  & 9 & 637.4  & 5 & 1993.4 & 3 & 29.9   & 4 & 1320.1 \\
      \textsf{nobel-us} & 9 & 1.5   & 9 & 67.1   & 4 & 1094.1 & 3 & 88.3   & 3 & 67.7 \\
      \textsf{pdh}      & 9 & 0.6   & 9 & 7.2    & 6 & 48.5   & 3 & 17.0   & 3 & 43.8 \\
      \bottomrule
  \end{tabular}
\fi
\end{table}

In Table~\ref{tab:sndlib_func_combo}, we additionally report results
disaggregated by travel cost and latency function.
For each function combination, we consider a total of $20$~\textsf{SNDlib}
instances, aggregated across all networks.
Overall, $125$ out of $180$~instances (\SI{69.4}{\percent}) are solved to
optimality.
Moreover, the results for the \textsf{SNDlib} instances support the previous
observations made for the Sioux Falls subnetworks.
All $20$ instances with {linear} travel costs are solved to optimality,
whereas the success rate decreases to \SI{65}{\percent} for {quadratic}
and \SI{43.3}{\percent} for {BPR} costs.
Among the latency functions, \textsf{sum\_ratio} tends to require longer
execution times when combined with {quadratic} costs
(\SI{916.9}{\second} on average), whereas \textsf{max\_ratio} remains faster
across all cost functions.

\begin{table}
  \centering
  \caption{The number of instances solved to optimality (``\#opt'', out of
    $20$ instances) and the average execution time (``time'', in
    \si{\second}) for different combinations of travel cost and latency
    functions.}
  \label{tab:sndlib_func_combo}
  \begin{tabular}{llrrrrr}
    \toprule
    $c(\flow)$ & $\latency(\flow)$ & \#opt & time \\
    \midrule
    \textsf{linear} & \textsf{sum\_ratio} & 20 & 39.3  \\
               & \textsf{max\_ratio} & 20 & 24.8  \\
               & \textsf{BPR}        & 20 & 22.2  \\
    \midrule
    \textsf{quadratic} & \textsf{sum\_ratio} & 14 & 916.9 \\
               & \textsf{max\_ratio} & 12 & 29.0  \\
               & \textsf{BPR}        & 13 & 538.7 \\
    \midrule
    \textsf{BPR} & \textsf{sum\_ratio} &  8 & 70.0  \\
               & \textsf{max\_ratio} &  9 & 577.7 \\
               & \textsf{BPR}        &  9 & 47.5  \\
    \bottomrule
  \end{tabular}
\end{table}

Overall, the \textsf{SNDlib} instances are more challenging to solve than the
Sioux Falls subnetworks, especially for nonlinear travel cost functions.
This is likely due to the more complex topologies of the \textsf{SNDlib}
networks, as the absence of articulation nodes limits the
effectiveness of certain enhancement techniques described in
Section~\ref{sec:tightening}.

\subsection{Managerial Insights}
\label{sec:managerial-insights}

To conclude this computational study, we highlight insights relevant
for practitioners and network operators.

When selecting travel cost functions, congestion measures, routing policies,
and uncertainty sets for traffic networks under demand uncertainty,
practitioners must balance model realism, computational tractability, and
solution quality.
Our study shows that the choice of travel cost function has a stronger impact
on computational tractability than the choice of congestion measure.
{Linear} cost functions yield the fastest solution times and highest
success rates. {BPR} cost functions, although more realistic in
capturing congestion effects, require longer solution times and may fail to
reach optimality for large instances. For operational settings that require
frequent re-optimization, {linear} or {quadratic} cost functions
provide a practical compromise between realism and computational tractability.
Nevertheless, flows obtained using BPR costs produce congestion levels that are
close to the worst case using {linear} or {quadratic} cost
functions across all latency measures, suggesting that {BPR} costs may serve
as a prudent choice if the appropriate measure is not known exactly.

Among the congestion measures, \textsf{sum\_ratio} seems to provide the best
trade-off between computational tractability and solution quality. The
\textsf{max\_ratio} measure, although effective at limiting extreme arc
utilization, leads to more challenging optimization problems if combined with
nonlinear travel cost functions. 

Comparing UE and SO formulations, we find that solving the UE
formulation requires only \SI{40}{\percent}--\SI{60}{\percent} of the time
needed for the SO formulation. Hence, the UE formulation is preferable for
repeated re-optimization and when a more realistic representation of travelers'
behavior is~desired.

Our study further underscores the importance of accounting for demand
uncertainty. Even the smallest uncertainty set considered in our evaluations
increases worst-case congestion by at least \SI{4.7}{\percent} compared to the
deterministic case. On average, solutions under uncertainty exhibit
\SI{11}{\percent}--\SI{72}{\percent} higher congestion than their deterministic
counterparts, depending on the chosen uncertainty set.
Ignoring uncertainty can thus lead to substantial underestimation of network
congestion.
Regarding the choice of the uncertainty set, we observe a trade-off between
worst-case estimates and computational tractability. 
%
The budgeted uncertainty set offers the most favorable balance, providing
protection against demand fluctuations with moderate increases in worst-case
congestion, while achieving high success rates and the fastest computation
times.
The ellipsoidal uncertainty set with~$\rho = \sqrt{\Gamma}$ yields similar
worst-case congestion levels as the budgeted model but is computationally
slightly more expensive. Thus, the budgeted model should be preferred.
Overall, the hose model and the ellipsoidal uncertainty set with~$\rho =
\Gamma$ produce the highest worst-case congestion levels.
Hence, these models are suitable if strong protection against fluctuating
demand is critical, but practitioners should be aware of the reduced
computational tractability of these formulations.


To sum up, for typical operational planning, the combination of a budgeted
uncertainty model with the \textsf{sum\_ratio} congestion measure and
{linear} or {quadratic} cost functions appears preferable, as it
provides resilient solutions within reasonable computation times for networks
with up to $100$~commodities.
If peak congestion is the main concern, the \textsf{max\_ratio} measure should
be adopted, keeping in mind that this may lead to more challenging problems to
solve. For long-term planning, where computation time is less critical,
{BPR} cost functions offer a more realistic representation of
congestion.


\section{Conclusion}
\label{sec:conclusion}

We study the problem of determining the worst-case congestion in
a multi-commodity traffic network under demand uncertainty. Our aim is to
stress-test a given network by identifying demand realizations and
corresponding travelers' route choices that maximize congestion. To this end,
we propose a novel bilevel formulation in which a traffic planner acts as the
leader and the users of the traffic network act as the followers.
In the leader's problem, we account for demand uncertainty using ideas from
robust optimization, whereas the follower's problem models a traffic
equilibrium in which travelers follow one of the two Wardrop
principles---the user equilibrium (UE) or the system optimum (SO).
We develop single-level mixed-integer nonlinear reformulations
for the resulting congestion models that exploit binary variables and big-$M$
constants, prove the existence of optimal solutions, derive valid big-$M$s, and
propose several enhancement techniques to strengthen the formulations.
Moreover, we perform an extensive computational study on instances of the Sioux
Falls network and the \textsf{SNDlib}, which provides several key insights.

First, we observe that the proposed enhancement techniques significantly
improve computational performance, reducing solution times by over
\SI{90}{\percent} and enabling the solution of larger network
instances that would otherwise remain unsolved within the time limit
of~\SI{2}{\hour}.
Second, the choice of travel cost and latency functions has a substantial
impact on both the computational tractability of the congestion models and the
estimated worst-case congestion.
In particular, we observe that the choice of the cost
function has a stronger impact on the tractability of the model than the
congestion measure.
%
Third, solving the UE formulation can be done considerably faster than solving
the SO formulation, whereas congestion levels under the UE are typically very
close to those under the SO. Indeed, by studying the so-called congestion
ratio, we show that the two equilibrium concepts yield similar worst-case
congestion levels.
This indicates that, under demand uncertainty, centralizing routing decisions
as in the SO does not provide a clear advantage over the UE in terms of
reducing congestion.
In particular, although the SO aims at minimizing the total travel time in the
system, it does not always lead to less congestion compared to the UE in our
framework.
Fourth and finally, our analysis of budgeted, ellipsoidal, and hose uncertainty
sets reveals a trade-off between worst-case congestion and computational
cost. The budgeted model can be solved the fastest, whereas the ellipsoidal and
hose models require more computation time but generally yield higher congestion
estimates.
Based on these findings, we derive managerial guidelines for choosing the most
appropriate modeling of congestion, travel costs, and uncertainty.

This work provides a first step towards designing traffic networks that
are more resilient to congestion under demand uncertainty. Building on our
stress-testing framework, a natural next step is to embed it into a
network design model, resulting in a robust bilevel optimization problem in
which infrastructure decisions are made while anticipating traffic equilibrium
responses under demand uncertainty.


\section*{Acknowledgment}

This research was funded in part by the French National Research
Agency (ANR) under project ANR-24-CE48-3565-03.
Sara Mattia is a member of GNAMPA-INdAM.
The authors contributed equally to this work (are co-lead authors) and are
listed in alphabetical order.


\enlargethispage{1cm}
\printbibliography

@article{Cerulli_et_al:2024,
  author       = {Martina Cerulli and
                  Claudia Archetti and
                  Elena Fern{\'{a}}ndez and
                  Ivana Ljubi\'c},
  title        = {A Bilevel Approach for Compensation and Routing Decisions in Last-Mile
                  Delivery},
  journal      = {Transportation Science},
  volume       = {58},
  number       = {5},
  pages        = {1076--1100},
  year         = {2024},
  doi          = {10.1287/TRSC.2023.0129}
}

@article{Cominetti_et_al:2024,
  title =	 {Monotonicity of equilibria in nonatomic congestion games},
  journal =	 {European Journal of Operational Research},
  volume =	 {316},
  number =	 {2},
  pages =	 {754--766},
  year =	 {2024},
  doi =		 {10.1016/j.ejor.2024.01.050},
  author =	 {Roberto Cominetti and Valerio Dose and Marco Scarsini},
}

@article{Tian_et_al:2017,
  author =	 {Tian, Liang and Bashan, Amir and Shi, Da-Ning and Liu,
                  Yang-Yu},
  year =	 {2017},
  title =	 {Articulation points in complex networks},
  journal =	 {Nature Communications},
  volume =	 {8},
  number =	 {14223},
  doi =		 {10.1038/ncomms14223},
}

@article{Hansen_et_al:1992,
  author       = {Hansen, Pierre and Jaumard, Brigitte and Savard, Gilles},
  publisher    = {SIAM},
  date         = {1992},
  doi          = {10.1137/0913069},
  journaltitle = {{SIAM} Journal on Scientific and Statistical Computing},
  number       = {5},
  pages        = {1194--1217},
  title        = {New branch-and-bound rules for linear bilevel programming},
  volume       = {13},
}

@article{Rey_Levin:2024,
  title =	 {A branch-and-price-and-cut algorithm for discrete network
                  design problems under traffic equilibrium},
  journal =	 {Transportation Research Part B: Methodological},
  volume =	 {212},
  pages =	 {103563},
  year =	 {2026},
  author =	 {David Rey and Michael W. Levin},
  doi =		 {10.1016/j.trb.2026.103563},
}

@incollection{Correa_Stier-Moses:2011,
  title =	 {Wardrop Equilibria},
  author =	 {José R. Correa and Nicolás E. Stier-Moses},
  year =	 {2011},
  doi =		 {10.1002/9780470400531.eorms0962},
  booktitle =	 {Wiley Encyclopedia of Operations Research and Management
                  Science},
  publisher =	 {Wiley},
}

@book{Patriksson:2015,
  title =	 {The Traffic Assignment Problem: Models and Methods},
  author =	 {Patriksson, Michael},
  year =	 {2015},
  publisher =	 {Courier Dover Publications}
}

@misc{Boyles_et_al:2025,
  title =	 {Transportation Network Analysis, Volume I: Static and Dynamic
                  Traffic Assignment},
  author =	 {Stephen D. Boyles and Nicholas E. Lownes and Avinash
                  Unnikrishnan},
  year =	 {2025},
  url =		 {https://arxiv.org/abs/2502.05182},
}

@book{West:2001,
  author =	 {Douglas Brent West},
  title =	 {Introduction to Graph Theory},
  edition =	 {2nd Edition},
  publisher =	 {Pearson Education, Inc.},
  year =	 {2001},
}

@manual{Bureau-Public-Roads:1964,
  author =	 {{U.S.\ Bureau of Public Roads}},
  title =	 {Traffic assignment manual},
  publisher =	 {U.S. Government Printing Office},
  address =	 {Washington, D.C.},
  year =	 {1964},
}

@article{Wardrop_Whitehead:1952,
  author =	 {Wardrop, J. G. and Whitehead, J. I.},
  title =	 {Correspondence. Some Theoretical Aspects of Road Traffic
                  Research},
  journal =	 {ICE Proceedings: Engineering Divisions},
  volume =	 {1},
  number =	 {5},
  year =	 {1952},
  pages =	 {767},
  doi =		 {10.1680/ipeds.1952.11362}
}

@article{Wardrop:1952,
  author =	 {Wardrop, J. G.},
  title =	 {Some Theoretical Aspects of Road Traffic Research},
  journal =	 {Proceedings of the Institution of Civil Engineers},
  volume =	 {1},
  number =	 {3},
  year =	 {1952},
  pages =	 {325–-362},
  doi =		 {10.1680/ipeds.1952.11259},
}

@article{Soyster:1973,
  author =	 {Soyster, A. L.},
  date =	 {1973},
  doi =		 {10.1287/opre.21.5.1154},
  journaltitle = {Operations Research},
  number =	 {5},
  pages =	 {1154--1157},
  title =	 {Technical Note---Convex Programming with Set-Inclusive
                  Constraints and Applications to Inexact Linear Programming},
  volume =	 {21},
}

@book{Ben-Tal_et_al:2009,
  author       = {Ben-Tal, Aharon and Ghaoui, Laurent and Nemirovski, Arkadi},
  date         = {2009},
  doi          = {10.1515/9781400831050},
  journaltitle = {Robust Optimization},
  title        = {Robust Optimization},
}

@article{Bertsimas_et_al:2011,
  author =	 {Bertsimas, Dimitris and Brown, David B. and Caramanis,
                  Constantine},
  title =	 {Theory and Applications of Robust Optimization},
  journal =	 {SIAM Review},
  volume =	 {53},
  number =	 {3},
  pages =	 {464--501},
  year =	 {2011},
  doi =		 {10.1137/080734510},
}

@article{Chouman_et_al:2017,
  author =	 {Mervat Chouman and Teodor Gabriel Crainic and Bernard
                  Gendron},
  title =	 {Commodity Representations and Cut-Set-Based Inequalities for
                  Multicommodity Capacitated Fixed-Charge Network Design},
  year =	 {2017},
  journal =	 {Transportation Science},
  volume =	 {51},
  number =	 {2},
  pages =	 {650--667},
  doi =		 {10.1287/trsc.2015.0665},
}

@article{Bienstock_et_al:1998,
  author =	 {Bienstock, Daniel and Chopra, Sunil and G\"unl\"uk, Oktay and
                  Tsai, Chih-Yang},
  title =	 {Minimum cost capacity installation for multicommodity network
                  flows},
  journal =	 {Mathematical Programming},
  volume =	 {81},
  number =	 {2},
  pages =	 {177--199},
  year =	 {1998},
  doi =		 {10.1007/BF01581104},
}

@techreport{Beck_et_al:2024,
  author =	 {Yasmine Beck and Martine Labbé and Martin Schmidt},
  title =	 {Toll Setting with Robust Wardrop Equilibrium Conditions Under
                  Budgeted Uncertainty},
  year =	 {2024},
  url =		 {https://optimization-online.org/?p=26949}
}

@article{Ferris_Pang:1997,
  author =	 {Ferris, M. C. and Pang, J. S.},
  title =	 {Engineering and Economic Applications of Complementarity
                  Problems},
  journal =	 {SIAM Review},
  volume =	 {39},
  number =	 {4},
  pages =	 {669--713},
  year =	 {1997},
  doi =		 {10.1137/S0036144595285963},
}

@book{Luo_et_al:1996,
  author =	 {Zhi-Quan Luo and Jong-Shi Pang and Daniel Ralph},
  title =	 {Mathematical Programs with Equilibrium Constraints},
  year =	 {1996},
  doi =		 {10.1017/CBO9780511983658},
  publisher =	 {Cambridge University Press},
}

@article{Bertsimas_Sim:2003,
  author       = {Bertsimas, Dimitris and Sim, Melvyn},
  year         = {2003},
  doi          = {10.1007/s10107-003-0396-4},
  journal      = {Mathematical Programming},
  pages        = {49--71},
  title        = {Robust discrete optimization and network flows},
  volume       = {98},
}

@article{Bertsimas_Sim:2004,
  author       = {Bertsimas, Dimitris and Sim, Melvyn},
  year         = {2004},
  doi          = {10.1287/opre.1030.0065},
  journal      = {Operations Research},
  number       = {1},
  pages        = {35--53},
  title        = {The Price of Robustness},
  volume       = {52},
}

@thesis{Sim:2004,
  author      = {Sim, Melvyn},
  institution = {Massachusetts Institute of Technology, Sloan School of Management},
  url         = {https://dspace.mit.edu/handle/1721.1/17725},
  date        = {2004},
  title       = {Robust Optimization},
  type        = {phdthesis},
}

@article{Mattia:2019,
  author =	 {Sara Mattia},
  year =	 {2019},
  title =	 {A polyhedral analysis of the capacitated edge activation
                  problem with uncertain demands},
  journal =	 {Networks},
  volume =	 {74},
  number =	 {2},
  pages =	 {190--204},
  doi =		 {10.1002/net.21888},
}

@article{Mattia_Poss:2018,
  author =	 {Mattia, Sara and Poss, Michael},
  year =	 {2018},
  title =	 {A comparison of different routing schemes for the robust
                  network loading problem: polyhedral results and computation},
  journal =	 {Computational Optimization and Applications},
  volume =	 {69},
  number =	 {3},
  pages =	 {753--800},
  doi =		 {10.1007/s10589-017-9956-z},
}

@InProceedings{Ordonez_Stier-Moses:2007,
  author =	 {Ord{\'o}{\~{n}}ez, Fernando and Stier-Moses, Nicol{\'a}s E.},
  editor =	 {Chahed, Tijani and Tuffin, Bruno},
  title =	 {Robust Wardrop Equilibrium},
  booktitle =	 {Network Control and Optimization},
  year =	 {2007},
  publisher =	 {Springer Berlin Heidelberg},
  address =	 {Berlin, Heidelberg},
  pages =	 {247--256},
  isbn =	 {978-3-540-72709-5},
  doi =		 {10.1007/978-3-540-72709-5_26},
}

@article{Ordonez_Stier-Moses:2010,
  author =	 {Ord{\'o}{\~{n}}ez, Fernando and Stier-Moses, Nicol{\'a}s E.},
  title =	 {Wardrop Equilibria with Risk-Averse Users},
  journal =	 {Transportation Science},
  volume =	 {44},
  number =	 {1},
  pages =	 {63--86},
  year =	 {2010},
  doi =		 {10.1287/trsc.1090.0292},
}

@thesis{Ito:2011,
  title =	 {Robust Wardrop Equilibria in the Traffic Assignment Problem
                  with Uncertain Data},
  author =	 {Ito, Yoshihiko},
  year =	 {2011},
  school =	 {Graduate School of Informatics, Kyoto University},
  type =	 {mathesis},
  department =	 {Department of Applied Mathematics and Physics},
  url =
                  {http://www-optima.amp.i.kyoto-u.ac.jp/papers/master/2011_master_ito.pdf},
}

@article{Fortuny-Amat_Mccarl:1981,
  author       = {Fortuny-Amat, José and McCarl, Bruce},
  publisher    = {Palgrave Macmillan Journals on behalf of the Operational Research Society},
  year         = {1981},
  doi          = {10.1057/jors.1981.156},
  journal      = {The Journal of the Operational Research Society},
  number       = {9},
  pages        = {783--792},
  title        = {A Representation and Economic Interpretation of a Two-Level Programming Problem},
  volume       = {32},
}

@book{Dempe:2002,
  author    = {Dempe, S.},
  publisher = {Springer US},
  date      = {2002},
  doi       = {10.1007/b101970},
  title     = {Foundations of Bilevel Programming},
}

@article{Colson-et-al:2007,
  author =	 {Colson, Benoît and Marcotte, Patrice and Savard, Gilles},
  publisher =	 {Springer},
  date =	 {2007},
  journaltitle = {Annals of Operations Research},
  pages =	 {235--256},
  title =	 {An overview of bilevel optimization},
  volume =	 {153},
  doi =		 {10.1007/s10479-007-0176-2},
}

@article{Colson_et_al:2005,
  author =	 {Colson, Benoît and Marcotte, Patrice and Savard, Gilles},
  date =	 {2005},
  journaltitle = {4OR},
  pages =	 {87--107},
  title =	 {Bilevel programming: A survey},
  doi =		 {10.1007/s10288-005-0071-0},
}

@book{Ahuja_et_al:1993,
  author =	 {Ahuja, Ravindra K. and Magnanti, Thomas L. and Orlin, James
                  B.},
  title =	 {Network Flows: Theory, Algorithms, and Applications},
  year =	 {1993},
  publisher =	 {Prentice Hall},
}

@article{LeBlanc_et_al:1975,
  title =	 {An efficient approach to solving the road network equilibrium
                  traffic assignment problem},
  journal =	 {Transportation Research},
  volume =	 {9},
  number =	 {5},
  pages =	 {309--318},
  year =	 {1975},
  doi =		 {10.1016/0041-1647(75)90030-1},
  author =	 {Larry J. LeBlanc and Edward K. Morlok and William
                  P. Pierskalla},
}

@book{Bard:1998,
  author =	 {Bard, J. F.},
  publisher =	 {Springer New York, NY},
  date =	 {1998},
  doi =		 {10.1007/978-1-4757-2836-1},
  title =	 {Practical Bilevel Optimization: Algorithms and Applications},
  series =	 {Nonconvex Optimization and Its Applications},
}

@book{Dempe_et_al:2015,
  author =	 {Stephan Dempe and Vyacheslav Kalashnikov and Gerardo
                  A. Pérez-Valdés and Nataliya Kalashnykova},
  publisher =	 {Springer Berlin, Heidelberg},
  date =	 {2015},
  doi =		 {10.1007/978-3-662-45827-3},
  title =	 {Bilevel Programming Problems: Theory, Algorithms and
                  Applications to Energy Networks},
  series =	 {Energy Systems},
}

@article{Kleinert_et_al:2021,
  title =	 {A Survey on Mixed-Integer Programming Techniques in Bilevel
                  Optimization},
  author =	 {Thomas Kleinert and Martine Labbé and Ivana Ljubi\'c and
                  Martin Schmidt},
  year =	 {2021},
  journal =	 {EURO Journal on Computational Optimization},
  issn =	 {2192-4406},
  volume =	 {9},
  doi =		 {10.1016/j.ejco.2021.100007},
}

@inproceedings{SNDlib10,
  author =	 {S. Orlowski and M. Pi{\'o}ro and A. Tomaszewski and
                  R. Wess{\"a}ly},
  booktitle =	 {Proceedings of the 3rd International Network Optimization
                  Conference (INOC 2007), Spa, Belgium},
  title =	 {{SNDlib} 1.0--{S}urvivable {N}etwork {D}esign {L}ibrary},
  year =	 {2007},
  url =
                  {https://opus4.kobv.de/opus4-zib/frontdoor/deliver/index/docId/958/file/ZR-07-15.pdf}
}

@article{Orlowski_et_al:2010,
  author =	 {S. Orlowski and M. Pi{\'o}ro and A. Tomaszewski and
                  R. Wess{\"a}ly},
  doi =		 {10.1002/net.20371},
  journal =	 {Networks},
  number =	 {3},
  pages =	 {276--286},
  title =	 {{SNDlib} 1.0--{S}urvivable {N}etwork {D}esign {L}ibrary},
  volume =	 {55},
  year =	 {2010},
}

@article{Dempe_Zemkoho:2012,
  author =	 {Dempe, Stephan and Zemkoho, Alain},
  title =	 {Bilevel road pricing: theoretical analysis and optimality
                  conditions},
  journal =	 {Annals of Operations Research},
  volume =	 {196},
  pages =	 {223--240},
  year =	 {2012},
  doi =		 {10.1007/s10479-011-1023-z},
}

@article{LeBlanc_Boyce:1986,
  title =	 {A bilevel programming algorithm for exact solution of the
                  network design problem with user-optimal flows},
  journal =	 {Transportation Research Part B: Methodological},
  volume =	 {20},
  number =	 {3},
  pages =	 {259--265},
  year =	 {1986},
  doi =		 {10.1016/0191-2615(86)90021-4},
  author =	 {Larry J. LeBlanc and David E. Boyce}
}

@article{Marcotte:1986,
  author =	 {Marcotte, Patrice},
  year =	 {1986},
  title =	 {Network design problem with congestion effects: A case of
                  bilevel programming},
  journal =	 {Mathematical Programming},
  volume =	 {34},
  number =	 {2},
  pages =	 {142--162},
  doi =		 {10.1007/BF01580580},
}

@article{Ben-Ayed_et_al:1988,
  author =	 {Omar Ben-Ayed and David E. Boyce and Charles E. Blair},
  title =	 {A general bilevel linear programming formulation of the
                  network design problem},
  journal =	 {Transportation Research Part B: Methodological},
  volume =	 {22},
  number =	 {4},
  pages =	 {311--318},
  year =	 {1988},
  doi =		 {10.1016/0191-2615(88)90006-9},
}

@article{Migdalas:1995,
  author =	 {Migdalas, Athanasios},
  year =	 {1995},
  title =	 {Bilevel programming in traffic planning: Models, methods and
                  challenge},
  journal =	 {Journal of Global Optimization},
  volume =	 {7},
  number =	 {4},
  pages =	 {381--405},
  doi =		 {10.1007/BF01099649},
}

@article{Fontaine_Minner:2014,
  author =	 {Pirmin Fontaine and Stefan Minner},
  title =	 {Benders Decomposition for Discrete–Continuous Linear Bilevel
                  Problems with application to traffic network design},
  journal =	 {Transportation Research Part B: Methodological},
  volume =	 {70},
  pages =	 {163--172},
  year =	 {2014},
  doi =		 {10.1016/j.trb.2014.09.007},
}

@article{Labbe_et_al:1998,
  author =	 {Martine Labbé and Patrice Marcotte and Gilles Savard},
  title =	 {A Bilevel Model of Taxation and Its Application to Optimal
                  Highway Pricing},
  year =	 {1998},
  journal =	 {Management Science},
  volume =	 {44},
  number =	 {12.1},
  doi =		 {10.1287/mnsc.44.12.1608}
}

@article{Brotcorne_et_al:2001,
  author =	 {Luce Brotcorne and Martine Labbé and Patrice Marcotte and
                  Gilles Savard},
  title =	 {A Bilevel Model for Toll Optimization on a Multicommodity
                  Transportation Network},
  year =	 {2001},
  doi =		 {10.1287/trsc.35.4.345.10433},
  journal =	 {Transportation Science},
  volume =	 {35},
  number =	 {4},
}

@inbook{Kalashnikov_et_al:2020,
  author =	 {Kalashnikov, Vyacheslav and Flores Mu{\~{n}}iz, Jos{\'e}
                  Guadalupe and Kalashnykova, Nataliya},
  editor =	 {Kosheleva, Olga and Shary, Sergey P. and Xiang, Gang and
                  Zapatrin, Roman},
  title =	 {Bilevel Optimal Tolls Problems with Nonlinear Costs: A
                  Heuristic Solution Method},
  bookTitle =	 {Beyond Traditional Probabilistic Data Processing Techniques:
                  Interval, Fuzzy etc. Methods and Their Applications},
  year =	 {2020},
  publisher =	 {Springer International Publishing},
  address =	 {Cham},
  pages =	 {481--516},
  doi =		 {10.1007/978-3-030-31041-7_28},
}

@article{Dewez_et_al:2008,
  author =	 {Sophie Dewez and Martine Labbé and Patrice Marcotte and
                  Gilles Savard},
  title =	 {New formulations and valid inequalities for a bilevel pricing
                  problem},
  year =	 {2008},
  journal =	 {Operations Research Letters},
  volume =	 {36},
  number =	 {2},
  pages =	 {141--149},
  doi =		 {10.1016/j.orl.2007.03.005},
}

@article{Li_et_al:2011,
  author =	 {Li, Zukui and Ding, Ran and Floudas, Christodoulos A.},
  title =	 {A Comparative Theoretical and Computational Study on Robust
                  Counterpart Optimization: I. Robust Linear Optimization and
                  Robust Mixed Integer Linear Optimization},
  journal =	 {Industrial \& Engineering Chemistry Research},
  volume =	 {50},
  number =	 {18},
  pages =	 {10567--10603},
  year =	 {2011},
  doi =		 {10.1021/ie200150p},
}

@article{Altin:2007,
  author =	 {A. Alt{\i}n and E. Amaldi and P. Belotti and
                  M.\c{C}. P{\i}nar},
  title =	 {Provisioning virtual private networks under traffic
                  uncertainty},
  journal =	 {Networks},
  volume =	 {49},
  number =	 {1},
  year =	 {2007},
  pages =	 {100--115},
  doi =		 {10.1002/net.20145}
}

@article{mattiarnl,
  author =	 {S. Mattia},
  journal =	 {Computational Optimization and Applications},
  title =	 {The Robust Network Loading Problem with Dynamic Routing},
  volume =	 54,
  number =	 3,
  pages =	 {619--643},
  year =	 {2013},
  doi =		 {10.1007/s10589-012-9500-0},
}

@article{oriolo,
  author =	 {C. Chekuri and G. Oriolo and M.G. Scutell\`{a} and
                  B. Shepherd},
  journal =	 {Networks},
  number =	 {1},
  pages =	 {50--154},
  title =	 {Hardness of Robust Network Design},
  volume =	 {50},
  year =	 {2007},
  doi =		 {10.1002/net.20165},
}

@article{hose,
  title =	 {Designing Least-Cost Nonblocking Broadband Networks},
  journal =	 {Journal of Algorithms},
  volume =	 {24},
  number =	 {2},
  pages =	 {287--309},
  year =	 {1997},
  doi =		 {10.1006/jagm.1997.0866},
  author =	 {J.Andrew Fingerhut and Subhash Suri and Jonathan S. Turner},
}

@inproceedings{hose2,
  author =	 {Duffield, N. G. and Goyal, Pawan and Greenberg, Albert and
                  Mishra, Partho and Ramakrishnan, K. K. and van der Merive,
                  Jacobus E.},
  title =	 {A flexible model for resource management in virtual private
                  networks},
  year =	 {1999},
  publisher =	 {Association for Computing Machinery},
  address =	 {New York, NY, USA},
  doi =		 {10.1145/316188.316209},
  booktitle =	 {Proceedings of the Conference on Applications, Technologies,
                  Architectures, and Protocols for Computer Communication},
  pages =	 {95–108},
  location =	 {Cambridge, Massachusetts, USA},
  series =	 {SIGCOMM '99},
}

@inproceedings{PoA,
  author =	 {Koutsoupias, Elias and Papadimitriou, Christos},
  title =	 {Worst-case equilibria},
  year =	 {1999},
  publisher =	 {Springer-Verlag},
  address =	 {Berlin, Heidelberg},
  booktitle =	 {Proceedings of the 16th Annual Conference on Theoretical
                  Aspects of Computer Science},
  pages =	 {404--413},
  numpages =	 {10},
  location =	 {Trier, Germany},
  series =	 {STACS'99},
  doi =		 {10.1007/3-540-49116-3_38},
}

\end{document}